\documentclass[11pt]{amsart}

\usepackage{amsmath,amssymb,amsthm,mathtools}
\usepackage[margin=1.12in]{geometry}
\usepackage{microtype}
\usepackage{enumitem}
\usepackage[hidelinks]{hyperref}
\hypersetup{pdftitle={Real Inertia, Phase Obstructions, and de Rham Algebras for Orientifolds},pdfauthor={Xiaobin Li},pdfsubject={Journal replacement: theoretical proofs and the Hu-Wang complex stringy K-theory comparison}}
\usepackage[T1]{fontenc}
\usepackage{lmodern}

\numberwithin{equation}{section}

\newtheorem{theorem}{Theorem}[section]
\newtheorem{proposition}[theorem]{Proposition}
\newtheorem{lemma}[theorem]{Lemma}
\newtheorem{corollary}[theorem]{Corollary}

\theoremstyle{definition}
\newtheorem{definition}[theorem]{Definition}
\newtheorem{example}[theorem]{Example}
\theoremstyle{remark}
\newtheorem{remark}[theorem]{Remark}

\newcommand{\CR}{\mathrm{CR}}
\newcommand{\Ori}{\mathrm{Ori}}
\newcommand{\age}{\operatorname{age}}
\newcommand{\rk}{\operatorname{rk}}
\newcommand{\Fix}{\operatorname{Fix}}
\newcommand{\R}{\mathbb R}
\newcommand{\C}{\mathbb C}
\newcommand{\Z}{\mathbb Z}
\newcommand{\Ftwo}{\mathbb F_2}
\newcommand{\Aplus}{\mathcal A^{+}}

\newcommand{\AOri}{\mathcal A_{\Ori}}
\newcommand{\RR}{\mathfrak R}

\title{Real Inertia, Phase Obstructions, and de Rham Algebras for Orientifolds}
\author{Xiaobin Li}
\thanks{Corresponding author: \texttt{lixiaobin@home.swjtu.edu.cn}.}
\address{School of Mathematics, Southwest Jiaotong University, Chengdu, Sichuan, China}
\email{lixiaobin@home.swjtu.edu.cn}


\date{27 September 2026}

\begin{document}

\begin{abstract}
We study cohomological de Rham constructions for finite orientifold quotients. On the $G$-invariant even inertia, a phase-corrected anti-linear reflection induces a real form of the Chen-Ruan algebra. For anti-linear involutions of a complex line, we prove that no conjugacy-invariant single-sector age can satisfy the integer defect rule for all reflection pairs. For actions preserving a Calabi--Yau volume line, a twisted Bockstein class measures precisely the obstruction to a coherent real phase lift and to phase-trivial normalization.

We formulate an Euler-Gysin realization theorem for even-rank clean correspondence data with an explicit equality of twisted Euler classes. Stable real bundle information alone is not sufficient. For an effective phase-trivial generalized dihedral action on a Calabi--Yau surface, we construct an invariant rotated almost complex structure. The surface product is identified with the ordinary Chen--Ruan product for this rotated structure; all interactions involving an odd label have zero obstruction rank. This identifies the source of associativity and places the construction within existing orbifold theory. For an Eisenstein abelian surface, we compute the resulting $16$-dimensional invariant Frobenius algebra. Its multiplication is not supercommutative in the shifted grading. Further examples on square abelian and real quartic K3 surfaces exhibit disconnected or empty reflection loci and nonzero normal Euler classes. A torus-real K3 example has the same graded dimensions as the Eisenstein algebra but is distinguished by the fourth power of its degree one generator.
\end{abstract}

\maketitle

\section{Introduction and main results}

Chen--Ruan cohomology associates to an almost complex orbifold a graded cohomology theory built from the inertia orbifold, with multiplication corrected by an obstruction bundle \cite{ChenRuan}.  For global quotients the construction admits concrete descriptions \cite{FantechiGottsche,JKK}.  In the abelian case Chen and Hu gave a differential form realization in which the obstruction contribution is encoded by explicit twist factors \cite{ChenHu}.  Their construction provides a model of the Chen--Ruan ring in terms of differential forms on fixed loci. Kaufmann subsequently treated de Rham theory for arbitrary finite global quotients, relating the two orders of pullback and pushforward and working with the requisite identities up to homotopy \cite[Section~4]{KaufmannDeRham}. Jianxun Hu and Bai-Ling Wang give an intrinsic description of obstruction bundle for general almost complex orbifolds and a formal Thom realization of the Chen--Ruan product \cite[Theorems~3.2 and 3.8]{HuWang}. Thus the de Rham realization problem already has a substantial nonabelian literature.

Obstruction bundles and associative convolution products have likewise been studied systematically. The virtual intersection construction of Lupercio, Uribe and Xicot\'encatl already uses a triple excess criterion for associativity, following Fantechi--G\"ottsche \cite[Section~5]{LUXVirtual}. Hepworth gives a representation-theoretic description of the Chen--Ruan obstruction bundle, including noncommuting pairs \cite{Hepworth}. Edidin, Jarvis and Kimura construct orbifold products using logarithmic trace \cite{EJKLog}, and formulate sufficient excess bundle identities for inertial products \cite[Section~3.1]{EJKPlethora}. The Euler--Gysin theorem below is a real local coefficient version of this established pullback--Euler--pushforward mechanism, with a nonstable Euler comparison stated explicitly.

For broader background on orbifolds from groupoid and stringy topological viewpoints, we refer to \cite{AdemLeidaRuan,Moerdijk}.  The interpretation of twisted sectors through inertia and loop groupoids, together with their relation to gerbes, was developed in \cite{LupercioUribeLoop}.  Twisted $K$-theoretic refinements of the orbifold picture appear in \cite{AdemRuanTwistedK,LupercioUribeGerbes,TuXuLaurent}. Hu and Wang also construct a stringy product on the complexified orbifold $K$-theory of the base orbifold and identify it with Chen--Ruan cohomology by a modified delocalized Chern character \cite[Theorem~4.5]{HuWang}. For the rotated surface quotient in this paper, that theorem gives the ordinary complex $K$-theory comparison recorded in Corollary~\ref{cor:HuWang-rotation}; it is distinct from a Real $KR$-theoretic realization.

The purpose of this paper is to study how the Chen--Hu de Rham picture changes when orientation-reversing symmetries are introduced.  The natural group-theoretic datum is a finite $\Z_2$-graded group
\begin{equation}\label{eq:graded-group-intro}
  1\longrightarrow G\longrightarrow \widehat G
  \stackrel{\epsilon}{\longrightarrow}\Z_2\longrightarrow 1,
\end{equation}
where the even subgroup $G$ acts complex linearly and the odd elements act anti-complex linearly.  This is a natural global quotient model for orientifold geometry.  Graded groupoids, orientation-twisted transgression, Jandl structures and Real twistings occur naturally in related unoriented field theories and in Real $K$-theory \cite{YoungDW,NoohiYoung,HekmatiMurraySzaboVozzo,LudersOttoWaldorf}.

In the gerbe formulation of unoriented WZW and orientifold sigma models, orientation reversal is encoded by Jandl or equivariant Jandl structures; see \cite{SchreiberSchweigertWaldorf,GawedzkiSuszekWaldorfWZW,GawedzkiSuszekWaldorfGerbes}.  These constructions provide an important geometric precedent for the twisted phase and orientation data considered here.  The present construction, however, is different in nature: it is built directly from fixed-locus de Rham complexes, local coefficient systems and Euler--Gysin correspondences, rather than from gerbe holonomy.

There is an important distinction between even and odd fixed sectors.  If $g\in G$, then $X^g$ is an almost complex submanifold and carries the ordinary Chen--Ruan age.  If $\sigma\in\widehat G\setminus G$ is an anti-complex involution, then $X^\sigma$ is totally real and, in the anti-symplectic case, Lagrangian.  Thus its normal geometry is Real rather than a complex character representation.

A tempting strategy is to write an anti-linear normal line as
\begin{equation}\label{eq:tempting-weight}
  \sigma_\theta(z)=e^{2\pi i\theta}\overline z
\end{equation}
and to regard $\theta$ as an odd analogue of the Chen--Ruan age.  This is not intrinsic: all maps \eqref{eq:tempting-weight} are unitarily conjugate, whereas products of two such reflections are rotations with continuously varying ages.  The grading information in an odd--odd product is therefore relative.  It belongs naturally to a pair of Real boundary conditions; Maslov or Fredholm indices provide one analytic source of such relative data, but the cohomological constructions below are formulated independently of an analytic sewing theorem.

The paper develops this observation in five steps.

First, on the ordinary even inertia we construct a real form of the Chen--Hu algebra determined by the chosen reflection symmetry.  Fix an odd involution $\tau$ normalizing a finite abelian even subgroup $G$.  On a connected component $F\subset X^g$ define
\begin{equation}\label{eq:R-intro}
  \RR(\alpha_{g,F})
  =e^{\pi i\age(g,F)}
   (\tau^{-1}|_{\tau F})^*\overline{\alpha_{g,F}},
\end{equation}
with label transformation
\begin{equation}\label{eq:q-intro}
  q(g)=\tau g^{-1}\tau^{-1}.
\end{equation}
The inverse in \eqref{eq:q-intro} is the algebraic shadow of loop reversal.  The age phase in \eqref{eq:R-intro} compensates the orientation signs arising from the obstruction Euler class and the Gysin normal bundle.

\begin{theorem}[Closed-sector Real de Rham theorem]\label{thm:closed-real-intro}
Let $X$ be a compact almost complex manifold, let $G$ be a finite abelian group acting by almost complex automorphisms, and let $\tau$ be an anti-complex involution normalizing $G$. On the $G$-invariant shifted de Rham complex of the even inertia, \eqref{eq:R-intro} defines an anti-linear cochain involution. Its induced map on Chen--Ruan cohomology is an algebra automorphism, and its fixed cohomology is a real form of $H^*_{\CR}([X/G];\C)$. No strict associative multiplication on the raw differential forms is asserted.
\end{theorem}

Second, the usual rule of integer defect cannot be extended to all reflection pairs by a conjugacy-invariant single-sector fractional age.

\begin{theorem}[No canonical odd age]\label{thm:no-odd-age-intro}
There is no conjugacy-invariant function on anti-linear involutions of a complex line which simultaneously serves as a fractional sector age and makes the degree defect of every reflection--reflection product integer-valued.  In particular, a coordinate-independent odd Chen--Ruan age cannot be extracted from one anti-linear involution.
\end{theorem}

Third, in the Calabi--Yau setting the relative phase data assemble into a global obstruction class.  Let $\mathbb Z_\epsilon$ denote the sign module.  If the graded action preserves the complex line spanned by a holomorphic volume form $\Omega$, then its phase cocycle determines
\[
  [\mu_{\R}]\in H^2(\widehat G;\Z_\epsilon).
\]

\begin{theorem}[Global phase obstruction]\label{thm:phase-obstruction-intro}
The class $[\mu_{\R}]$ is, up to the standard sign convention, the Bockstein of the $\mathbb R/\mathbb Z$-valued phase cocycle.  The following are equivalent:
\begin{enumerate}[label=\textup{(\roman*)}]
\item $[\mu_{\R}]=0$;
\item the phase cocycle admits a coherent real lift;
\item after multiplying $\Omega$ by a constant unit complex number one can arrange
\[
 g^*\Omega=\Omega\quad(g\in G),
 \qquad
 \sigma^*\Omega=\overline\Omega\quad(\epsilon(\sigma)=1).
\]
\end{enumerate}
When the obstruction vanishes, coherent lifts form a torsor for $Z^1(\widehat G;\Z_\epsilon)$, and their equivalence classes form a torsor for $H^1(\widehat G;\Z_\epsilon)$.
\end{theorem}

Fourth, we isolate sufficient finite-dimensional Euler--Gysin data: correction bundles, coefficient lines, clean composition squares, and an actual comparison of their twisted Euler classes. The formal theorem is stated for even correction ranks, embedding codimensions and excess ranks, so that no unspecified odd-degree permutation signs enter the proof. Its conclusion is associativity on cohomology. In the surface case there is a more concrete explanation: a rotation of almost complex structure turns the entire graded group into a group of almost complex automorphisms.

\begin{theorem}[Calabi--Yau surface realization]\label{thm:CY2-intro}
Let $(X,\omega,J,\Omega)$ be a compact connected K\"ahler surface with a nowhere-vanishing holomorphic two-form. Let a finite abelian group $G$ act effectively preserving $\omega,J,\Omega$, and let $\tau$ be an anti-holomorphic anti-symplectic involution satisfying
\[
 \tau^*\Omega=\overline\Omega,\qquad \tau g\tau=g^{-1}.
\]
Set $\widehat G=G\rtimes\langle\tau\rangle$. There is a $\widehat G$-invariant almost complex structure $J_\Omega$, determined by the given metric and $\operatorname{Re}\Omega$, such that the proposed group-labelled surface product is the ordinary Chen--Ruan global quotient product for $(X,J_\Omega)$. Its $\widehat G$-invariant part is therefore a canonical unital associative graded algebra,
\[
 H^*_{\Ori}(X,\widehat G)\cong
 H^*_{\CR}([(X,J_\Omega)/\widehat G];\R).
\]
Nontrivial even sectors have shift $2$, odd sectors have shift $1$, and every binary product involving an odd label is the oriented clean intersection Gysin product with zero obstruction bundle. The even group-labelled restriction is the original Chen--Hu product.
\end{theorem}

Theorem~\ref{thm:phase-obstruction-intro} shows that the vanishing of $[\mu_{\R}]$ is equivalent to the existence of the phase-trivial normalization used in Theorem~\ref{thm:CY2-intro}.

Finally, we compute the resulting algebra in a compact example with nontrivial even and odd sectors.  Let $E_\omega=\C/(\Z+\omega\Z)$ for $\omega=e^{2\pi i/3}$ and put $X=E_\omega\times E_\omega$.  A generalized dihedral action generated by
\[
 r(z_1,z_2)=(\omega z_1,\omega^{-1}z_2),
 \qquad
 \tau(z_1,z_2)=(-\overline z_1,-\overline z_2)
\]
is phase-trivial.  After passing to conjugacy invariants, the orientifold state space is $16$-dimensional and its multiplication can be written explicitly.  In particular, the odd-sector multiplication is noncommutative in the shifted grading.

\begin{theorem}[Eisenstein orientifold algebra]\label{thm:eisenstein-intro}
For the action above, the centralizer-invariant state space is a $16$-dimensional associative Frobenius algebra.  With the notation of Section~\ref{sec:examples}, its multiplication is generated by
\[
 \Theta^2=6q,
 \qquad
 U_pU_{p'}=2\delta_{p,p'}q,
\]
\[
 E^2=3(W-\Theta),
 \qquad
 \Theta E=-6T,
 \qquad
 U_pE=2T,
\]
\[
 ET=TE=3q,
 \qquad
 AB=3q,
 \qquad
 BA=-3q,
\]
where $W=\sum_{p\in X^r}U_p$.  The trace pairing is nondegenerate.
\end{theorem}

The further examples in Sections~\ref{subsec:one-reflection}--\ref{subsec:quartic-k3} separate several phenomena that the Eisenstein computation alone does not test. Standard and affine conjugation on the same square abelian surface give algebras of dimensions $24$ and $8$, respectively. Three signed real quartics, all isomorphic over $\C$, give dimensions $15$, $16$, and $12$. The spherical real locus produces the nonzero normal-Euler relation $e^4=-2q$. The torus-real K3 algebra has the same graded dimensions as the Eisenstein algebra but has $e^4=0$, proving that the two graded algebras are not isomorphic.

\paragraph{Relation to existing orbifold and rotation constructions.}
Equivariant unoriented TQFTs and extended Frobenius algebras provide an established approach to unoriented operations; see Sweet \cite{Sweet} and Young \cite{YoungDW}. We do not construct or compare all of those operations here. Our formulas organize even and reflection fixed loci in a common group-labelled correspondence picture. In complex dimension two, however, this description does \emph{not} produce a cohomology ring outside ordinary Chen--Ruan theory: Theorem~\ref{thm:CY2-intro} identifies it with that theory after rotation.

The geometric rotation principle is also established. In the K3 setting it is commonly called Donaldson's trick; see \cite[Section~3]{YoshikawaRealK3} and \cite[Remark~2.3 and Proposition~2.4]{ItenbergMikhalkin} and the earlier references there. Biswas--Wilkin \cite{BiswasWilkin} show that a suitable anti-holomorphic involution on a hyper-K\"ahler manifold becomes holomorphic after rotation, and its fixed locus becomes complex Lagrangian. Below we give the elementary, pointwise almost complex construction needed for the whole finite group on a surface; integrability of the rotated structure is unnecessary for Chen--Ruan cohomology \cite{ChenRuan}. Thus the role of the surface theorem is an explicit identification and computation using existing structures, not the introduction of an independent generalized cohomology theory.

\paragraph{\textbf{Relation to companion analytic work.}}
A related companion manuscript of the author studies a different problem: the analytic Real Cauchy--Riemann sewing underlying graded inertia composition, at the level of real family indices and a $KO$-valued excess identity.  That analytic theorem is not used in the principal results of the present paper.  In particular, Theorem~\ref{thm:phase-obstruction-intro} is a group-cohomological Bockstein statement, while Theorems~\ref{thm:CY2-intro} and \ref{thm:eisenstein-intro} are proved directly by finite-dimensional clean-intersection and de Rham/Gysin arguments.  The arguments below are self-contained relative to the cited ordinary orbifold results and do not depend on a comparison with the contents of that companion manuscript.

\paragraph{\textbf{Scope.}}
The general realization theorem requires explicit Euler class comparison, not merely equality of stable real bundles or family indices. Its even-rank restriction is a sufficient hypothesis covering the complex and surface examples treated here; an arbitrary odd-rank extension would require additional sign conventions and proofs. Phase normalization is not, in arbitrary dimension, an existence criterion for these Euler--Gysin data. The surface construction is canonically defined relative to the specified Calabi-Yau data and is identified with a rotated ordinary orbifold ring. We do not assert a strict chain level differential graded algebra, a universal higher-dimensional odd-sector product, a twisted $KR$-theoretic realization, or a Real Gromov-Witten construction.

The paper is organized as follows: Section~\ref{sec:inertia} develops the closed sector Real Chen-Hu model.  Section~\ref{sec:odd} proves the odd-age no-go theorem and the global phase obstruction theorem.  Section~\ref{sec:admissible} develops the general Euler-Gysin realization and proves the Calabi-Yau surface existence theorem.  Section~\ref{sec:examples} contains the dihedral examples, the explicit Eisenstein algebra, and five additional square torus and real quartic K3 examples.  Section~\ref{sec:comparison} compares the construction with previous approaches and discusses further directions.  Appendix~\ref{app:det} records determinant line and excess intersection conventions.

\section{Real inertia and the closed-sector de Rham model}\label{sec:inertia}

\subsection{Finite graded actions}

Let $(X,J)$ be a compact almost complex manifold and let \eqref{eq:graded-group-intro} be a finite graded group.  We assume
\begin{equation}\label{eq:parity-action}
  g_*J=Jg_*\quad(g\in G),\qquad
  \sigma_*J=-J\sigma_*\quad(\epsilon(\sigma)=1).
\end{equation}
The grading $\epsilon$ refers to anti-complex or orientifold reversal, not necessarily to reversal of the real orientation of $X$: an anti-complex map in complex dimension $n$ has real orientation sign $(-1)^n$. In particular, odd elements preserve the ambient orientation on a complex surface.

For most statements involving odd fixed loci we further assume that the odd element under consideration is an involution.  This reflection-type hypothesis is sufficient for the examples and avoids the additional one-sided holonomy carried by an odd element with nontrivial even square.

The unreduced Real inertia space is
\begin{equation}\label{eq:real-inertia}
  I_{\Ori}X=\coprod_{\gamma\in\widehat G}X^\gamma
  =I^+X\sqcup I^-X,
\end{equation}
where
\[
 I^+X=\coprod_{g\in G}X^g,\qquad
 I^-X=\coprod_{\epsilon(\sigma)=1}X^\sigma.
\]
For nonabelian groups the quotient-stack sector associated to $\gamma$ is $[X^\gamma/C_{\widehat G}(\gamma)]$.  We keep the unreduced notation when writing local correspondences and take centralizer invariants when passing to quotient-stack cohomology.

\begin{proposition}[Odd fixed loci]\label{prop:totally-real}
Let $\sigma$ be an odd involution.  Then $L_\sigma=X^\sigma$ is totally real:
\[
 T L_\sigma\cap J T L_\sigma=0.
\]
If $X$ is symplectic, $J$ is compatible, and $\sigma^*\omega=-\omega$, then $L_\sigma$ is Lagrangian.
\end{proposition}

\begin{proof}
If $v\in T_xL_\sigma$, then $\sigma_*v=v$.  By \eqref{eq:parity-action},
$\sigma_*(Jv)=-Jv$.  If $Jv$ were tangent to $L_\sigma$, it would also be fixed by $\sigma_*$, hence $Jv=-Jv$ and $v=0$.  This proves total reality.  In the anti-symplectic case, for $u,v\in T_xL_\sigma$,
\[
 \omega(u,v)=(\sigma^*\omega)(u,v)=-\omega(u,v),
\]
so $L_\sigma$ is isotropic.  Its real dimension is half the real dimension of $X$, and it is therefore Lagrangian.
\end{proof}

The proposition already shows a basic obstruction to copying Chen--Ruan theory verbatim: an odd fixed component does not inherit the complex normal eigenbundle decomposition from which the ordinary age is defined.

\begin{proposition}[Normal bundle of a reflection locus]\label{prop:normal-JTL}
For an odd involution $\sigma$, the almost complex structure induces a canonical real-bundle isomorphism
\[
 J:T L_\sigma\stackrel{\cong}{\longrightarrow}N_{L_\sigma/X}.
\]
In particular,
\[
 o(N_{L_\sigma/X})\cong o(TL_\sigma).
\]
\end{proposition}

\begin{proof}
By Proposition~\ref{prop:totally-real}, $T_xL_\sigma\cap JT_xL_\sigma=0$.  Since $L_\sigma$ has real dimension equal to $\dim_\C X$, one has the direct sum decomposition
\[
 T_xX=T_xL_\sigma\oplus JT_xL_\sigma.
\]
Projection to the quotient $T_xX/T_xL_\sigma$ identifies $JT_xL_\sigma$ with the normal space.  These identifications vary smoothly in $x$.
\end{proof}

Thus nonorientability of a Real fixed locus is reflected directly in its normal orientation local system.  For example, the standard real locus $\mathbb{RP}^2\subset\mathbb{CP}^2$ is nonorientable, while $\mathbb{RP}^1\subset\mathbb{CP}^1$ is orientable.  This is one reason that odd-sector coefficient systems must be treated geometrically rather than assigned by a universal sign convention.

\subsection{The even Chen--Hu complex}

We recall the part of the Chen--Hu model that will be used below.  Assume in this subsection that $G$ is finite abelian.  For a connected component $F\subset X^g$, let $\age(g,F)$ be the usual degree-shifting number.  Write
\begin{equation}\label{eq:CH-complex}
 \Aplus(X,G)
 =\left(\bigoplus_{g\in G}\ \bigoplus_{F\subset X^g}
 \Omega^{\bullet-2\age(g,F)}(F;\C)\right)^G.
\end{equation}
The shifts are rational and componentwise. Abelianity does not make the action on the fixed-locus cohomology trivial; the invariants in \eqref{eq:CH-complex} are essential. The product is first written on the unreduced group-labelled sum and then restricted to invariants \cite{FantechiGottsche}. We use the standard Chen--Hu/Chen--Ruan product on its \emph{cohomology} \cite{ChenHu,ChenRuan}. At the level of ordinary forms, a displayed Gysin expression means a representative constructed with specified Thom forms and tubular neighborhoods; its cohomology class is independent of those choices.  If $W\subset X^{g,h}$ is a connected common fixed component and
\[
 e_1:W\to F_g,\qquad e_2:W\to F_h,\qquad
 m:W\to F_{gh}
\]
are the natural maps, then the $W$-contribution is
\begin{equation}\label{eq:CH-product}
 \alpha_g\star_W\beta_h
 =m_*\bigl(e_1^*\alpha_g\wedge e_2^*\beta_h\wedge e(E_{g,h})\bigr).
\end{equation}
Let
\[
 r_W=\rk_\C E_{g,h}|_W,
 \qquad c_W=\operatorname{codim}_\C(W,F_{gh}).
\]
Degree compatibility gives the familiar identity
\begin{equation}\label{eq:age-rank}
 \age(g,F_g)+\age(h,F_h)-\age(gh,F_{gh})=r_W+c_W.
\end{equation}
For abelian global quotients this can also be verified line-by-line from the simultaneous eigenbundle decomposition; compare \cite{ChenHu}.

For general almost complex orbifolds, Hu--Wang \cite[Theorems~3.2 and 3.8]{HuWang} describe the complex obstruction class in rational orbifold $K$-theory and encode the Chen--Ruan product by a formal Thom pushforward. Their formal fractional Thom classes are not ordinary differential forms of fractional degree. This distinction is consistent with our use of \eqref{eq:CH-product} on cohomology and supplies no strict associative multiplication on the raw forms in \eqref{eq:CH-complex}.

\subsection{Weight bookkeeping in the Chen--Hu product}

We recall a local calculation which will be used repeatedly.  Let $W$ be a connected component of $X^{g,h}$ and decompose the normal representation over $W$ into simultaneous complex eigenbundles $V_j$.  Write the fractional weights of $g$ and $h$ on $V_j$ as $a_j,b_j\in[0,1)$.  The $gh$-weight is $\{a_j+b_j\}$.  There are three possibilities.
\begin{enumerate}[label=\textup{(\alph*)}]
\item If $a_j+b_j<1$, the direction contributes neither to the obstruction bundle nor to the Gysin normal bundle.
\item If $a_j+b_j=1$, the output $gh$ acts trivially in that direction.  The direction is normal to $W$ inside the output fixed locus and contributes one complex Gysin codimension.
\item If $a_j+b_j>1$, one integer carry occurs and the direction contributes one complex obstruction direction.
\end{enumerate}
Hence, line by line,
\[
 a_j+b_j-\{a_j+b_j\}
 =\begin{cases}
 0,&a_j+b_j<1,\\
 1,&a_j+b_j\geq1,
 \end{cases}
\]
with the equality case assigned to the Gysin normal direction and the strict inequality case to the obstruction direction.  Summing gives \eqref{eq:age-rank}.  This is the elementary bookkeeping behind both the degree shift and the sign cancellation in Theorem~\ref{thm:closed-real}.

The same calculation also explains why the reflection problem is different.  For an anti-linear involution there is no invariant fractional eigenvalue $a_j$ in $[0,1)$: a phase in an expression $z\mapsto e^{2\pi i\theta}\bar z$ can be removed by a unitary change of coordinates.  Thus the above line-by-line carry rule has no single-sector odd analogue.

\subsection{Reflection of closed sectors}

Fix an odd involution $\tau\in\widehat G\setminus G$.  Since $G\triangleleft\widehat G$, conjugation by $\tau$ preserves $G$.  Define
\begin{equation}\label{eq:q}
 q(g)=\tau g^{-1}\tau^{-1}.
\end{equation}
If $F\subset X^g$, then $\tau F\subset X^{q(g)}$.

\begin{lemma}[Age preservation]\label{lem:age-preservation}
For every connected component $F\subset X^g$,
\[
 \age(q(g),\tau F)=\age(g,F).
\]
\end{lemma}

\begin{proof}
Let $v$ be a complex normal eigenvector for $g$ with eigenvalue $e^{2\pi i\theta}$, $0\leq\theta<1$.  Then
\[
 q(g)d\tau(v)=d\tau(g^{-1}v)
 =d\tau(e^{-2\pi i\theta}v)
 =e^{2\pi i\theta}d\tau(v),
\]
because $d\tau$ is anti-complex linear.  Thus the fractional weight is preserved.  Summing the normal weights proves the claim.
\end{proof}

Define an anti-linear map on \eqref{eq:CH-complex} by
\begin{equation}\label{eq:R-def}
 \RR(\alpha_{g,F})
 =e^{\pi i\age(g,F)}
 (\tau^{-1}|_{\tau F})^*\overline{\alpha_{g,F}}.
\end{equation}
The target of \eqref{eq:R-def} is the $q(g),\tau F$ sector. Since $\tau$ normalizes $G$, this map carries $G$-invariant collections to $G$-invariant collections.

\begin{lemma}\label{lem:R-involution}
The map $\RR$ commutes with $d$ and satisfies $\RR^2=1$.
\end{lemma}

\begin{proof}
Commutation with $d$ follows from naturality of the exterior derivative and complex conjugation.  Since $\tau^2=1$, one has $q^2=1$.  By Lemma~\ref{lem:age-preservation} the two phase factors in $\RR^2$ are complex conjugates, hence cancel.
\end{proof}

The key point is multiplicativity on cohomology. We record the orientation sign explicitly; the following comparison is an identity of cohomology classes, or equivalently an identity of de Rham representatives modulo exact forms.

\begin{lemma}[Anti-complex sign]\label{lem:anti-complex-sign}
For a $W$-contribution \eqref{eq:CH-product}, the unphased anti-linear pullback
\[
 U(\alpha_{g,F})=(\tau^{-1}|_{\tau F})^*\overline{\alpha_{g,F}}
\]
satisfies
\begin{equation}\label{eq:U-sign}
 U(\alpha_g\star_W\beta_h)
 =(-1)^{r_W+c_W}\,
 (U\alpha_g\star_{\tau W}U\beta_h).
\end{equation}
\end{lemma}

\begin{proof}
The Real map induced by $\tau$ on a complex vector space of dimension $r_W$ reverses the complex orientation by $(-1)^{r_W}$.  Hence conjugating the obstruction Euler class contributes $(-1)^{r_W}$.  The Gysin map $m_*$ is defined by a Thom class of the complex normal bundle of $W$ in $F_{gh}$; anti-complex conjugation reverses its orientation by $(-1)^{c_W}$.  The remaining pullbacks and wedge products are natural.  Multiplying the two signs gives \eqref{eq:U-sign}.
\end{proof}

\begin{theorem}[Closed-sector Real de Rham theorem]\label{thm:closed-real}
The cochain involution $\RR$ induces an anti-linear involutive algebra automorphism of the Chen--Ruan cohomology algebra:
\begin{equation}\label{eq:R-multiplicative}
 \RR([\alpha]\star[\beta])=\RR[\alpha]\star\RR[\beta].
\end{equation}
The fixed cochain complex $\Aplus_{\R}(X,G,\tau)=\Fix(\RR)$ has an induced real cohomology algebra, and
\[
 H^*(\Aplus_{\R})\otimes_{\R}\C
 \cong H^*_{\CR}([X/G];\C)
\]
as complex algebras. This does not assert a strict algebra structure on the raw fixed differential forms.
\end{theorem}
\begin{proof}
For a connected $W$-contribution, Lemma~\ref{lem:anti-complex-sign} gives the phase on the left as
\[
 e^{\pi i\age(gh)}(-1)^{r_W+c_W}
 =e^{\pi i(\age(g)+\age(h))}
\]
by \eqref{eq:age-rank}. This is the product of the two input phases. Thus the induced map is multiplicative. Lemma~\ref{lem:R-involution} proves involutivity at the cochain level. Every complex cochain has the unique real-form decomposition
\[
 v=\frac{v+\RR v}{2}
   +i\frac{v-\RR v}{2i},
\]
and both summands after removing the factor $i$ are fixed by $\RR$. This decomposition commutes with the differential. Complexification therefore identifies the fixed cohomology with the full $G$-invariant Chen--Ruan cohomology, with its product.
\end{proof}

\begin{remark}\label{rem:closed-only}
Theorem~\ref{thm:closed-real} concerns the ordinary complex twisted sectors, equipped with the extra reflection symmetry induced by the orientifold.  It does not yet introduce odd fixed-locus states. Also, $\Fix(\RR)$ is a real form obtained using label inversion and complex conjugation; it is not the same operation as taking ordinary conjugation invariants under $\widehat G$. Those two constructions must not be identified without a separate comparison.  The next section explains the additional relative data needed for genuine odd interactions.
\end{remark}

\section{Odd sectors: a no-go theorem and the relative replacement}\label{sec:odd}

The normal geometry of an odd involution is not described by a complex character.  This elementary point is easy to miss because an anti-linear involution of a complex line can be written with a phase.  The phase is a coordinate choice, not an invariant.

\subsection{Anti-linear involutions of a complex line}

For $\theta\in\R/\Z$ set
\begin{equation}\label{eq:sigma-theta}
 \sigma_\theta(z)=e^{2\pi i\theta}\overline z.
\end{equation}
Every $\sigma_\theta$ is an anti-linear involution.

\begin{lemma}[Unitary conjugacy]\label{lem:unitary-conjugacy}
All anti-linear involutions \eqref{eq:sigma-theta} are unitarily conjugate.
\end{lemma}

\begin{proof}
Let $U_\theta(z)=e^{-\pi i\theta}z$.  Then
\[
 U_\theta\sigma_\theta U_\theta^{-1}(z)
 =U_\theta\bigl(e^{\pi i\theta}\overline z\bigr)
 =\overline z.
\]
\end{proof}

Thus a number extracted from $\theta$ alone cannot be invariant under complex-linear change of coordinates.  Relative phases, by contrast, are meaningful.  For two reflections,
\begin{equation}\label{eq:reflection-product}
 \sigma_0\sigma_\theta(z)=e^{-2\pi i\theta}z,
\end{equation}
which is an ordinary rotation.  If $0<\theta<1$, its complex age is $1-\theta$.

\subsection{No canonical odd age}

The ordinary age enters the Chen--Ruan grading through an integer defect: for even sectors, the difference of ages equals an obstruction rank plus a complex codimension.  One might ask for an analogous function on odd involutions.  The following proposition rules this out already in complex dimension one.

\begin{theorem}[No canonical odd age]\label{thm:no-odd-age}
Let $a$ be a real-valued function on anti-linear involutions of a complex line satisfying:
\begin{enumerate}[label=\textup{(\roman*)}]
\item $a$ is invariant under complex-linear conjugacy;
\item for every pair of anti-linear involutions $\sigma,\rho$, the quantity
\begin{equation}\label{eq:odd-defect}
 a(\sigma)+a(\rho)-\age(\sigma\rho)
\end{equation}
is integer-valued whenever $\sigma\rho\neq 1$.
\end{enumerate}
Then no such function exists.
\end{theorem}

\begin{proof}
By Lemma~\ref{lem:unitary-conjugacy}, condition (i) implies that $a(\sigma_\theta)=c$ is constant.  For $0<\theta<1$, \eqref{eq:reflection-product} has age $1-\theta$.  Condition (ii) would therefore require
\[
 2c-(1-\theta)\in\Z
\]
for every $\theta\in(0,1)$.  The left-hand side varies continuously and nontrivially with $\theta$, which is impossible for an integer-valued function.
\end{proof}

\begin{corollary}\label{cor:no-positive-part}
A formula such as $\sum_j\max(\theta_j,0)$, formed from phases in local expressions $z_j\mapsto e^{2\pi i\theta_j}\overline z_j$, is not a coordinate-invariant age of a single odd sector.
\end{corollary}

This observation changes the architecture of the orientifold problem.  Odd sectors can certainly carry ordinary or local-coefficient de Rham cohomology, but their multiplication cannot be graded by a fractional invariant assigned independently to each anti-linear fixed locus.

\subsection{Relative angle and pairwise grading data}

The invariant object attached to a pair of odd involutions is their relative position.  In the local model \eqref{eq:sigma-theta}, the fixed real lines are
\[
 L_0=\R,
 \qquad
 L_\theta=e^{\pi i\theta}\R.
\]
Their relative angle determines the rotation \eqref{eq:reflection-product}.  After a grading convention is fixed, the same relative angle also determines an integer Maslov index for the ordered pair.  In higher dimension, transverse Real fixed loci may be analyzed blockwise, so the pairwise grading contribution is additive over the corresponding two-dimensional blocks.

For the de Rham theory developed below, we use only this structural conclusion: the correction attached to an odd interaction is pairwise rather than a sum of invariants assigned separately to the two odd sectors.  Analytically, such pairwise data may be produced by Real Cauchy--Riemann/Fredholm theory and its determinant lines. For orientations of disc moduli under anti-symplectic involutions and the required relative spin data, see \cite{FOOOAnti}; for compatible Fredholm determinant line conventions, see \cite{ZingerDet}; and for Real Gromov--Witten orientations, see \cite{GeorgievaZinger}.  The analytic construction and its $KO$-valued sewing identity are treated separately and are not used in the finite-dimensional Calabi--Yau surface realization proved here.

From the analytic point of view, this relative integer is the local model for the Fredholm index of a Cauchy--Riemann operator with Lagrangian boundary conditions; the standard analytic framework may be found, for example, in \cite{McDuffSalamon}.  What is important for the present paper is that this index is attached to a pair of boundary conditions rather than to either Real sector separately.

\subsection{Odd de Rham spaces with local coefficients}

Let $\sigma$ be an odd involution and $L_\sigma=X^\sigma$.  A minimal de Rham state space on this locus is
\begin{equation}\label{eq:odd-derham}
 \Omega^*(L_\sigma;\mathcal L_\sigma),
\end{equation}
where $\mathcal L_\sigma$ is a rank-one real local system encoding the orientation data required by the intended Gysin maps.  There is no choice satisfying the universal single-sector defect rule for \eqref{eq:odd-derham}.  An integer shift may be supplied by a relative grading package, and this extra datum will be included explicitly in Definition~\ref{def:admissible-package}.

\begin{remark}
The conclusion is intentionally weaker than a universal no-go statement for orientifold cohomology.  Theorem~\ref{thm:no-odd-age} rules out one natural strategy: assigning a coordinate-invariant Chen--Ruan-style fractional age to each odd involution separately.  It does not rule out generalized cohomology, bivariant theories, $KR$-theory, or constructions whose grading is genuinely relative.
\end{remark}

\subsection{The global phase obstruction}\label{subsec:phase-obstruction}

Suppose now that $X$ is Calabi--Yau with nowhere-vanishing holomorphic volume form $\Omega$, and assume the graded action preserves the complex line spanned by $\Omega$.  Thus there are phases $a_\gamma\in\R/\Z$ such that
\[
 g^*\Omega=e^{2\pi i a_g}\Omega\quad(g\in G),
 \qquad
 \sigma^*\Omega=e^{2\pi i a_\sigma}\overline\Omega
 \quad(\epsilon(\sigma)=1).
\]
Let $\Z_\epsilon$, $\R_\epsilon$ and $(\R/\Z)_\epsilon$ denote the sign modules, with action
\[
 \gamma\cdot t=(-1)^{\epsilon(\gamma)}t.
\]
The phase law is
\begin{equation}\label{eq:phase-law}
 a_{\gamma\delta}
 \equiv a_\gamma+(-1)^{\epsilon(\gamma)}a_\delta
 \pmod{\Z},
\end{equation}
so $a$ is a $1$-cocycle with values in $(\R/\Z)_\epsilon$.

Choose real lifts $\widetilde a_\gamma$ and define
\begin{equation}\label{eq:maslov-cocycle}
 \mu_{\R}(\gamma,\delta)
 =\widetilde a_{\gamma\delta}-\widetilde a_\gamma
 -(-1)^{\epsilon(\gamma)}\widetilde a_\delta
 \in\Z.
\end{equation}
The phase law implies that $\mu_\R$ is a twisted $2$-cocycle.

Twisted group cohomology also appears in equivariant Jandl structures \cite{GawedzkiSuszekWaldorfWZW} and in orientifold discrete torsion, understood through equivariant $B$-field data \cite{SharpeDiscrete}. Here the additive notation $U(1)_\epsilon=\R/\Z$ uses the sign action. The volume phase is a class in $H^1(\widehat G;U(1)_\epsilon)$, and its Bockstein lies in $H^2(\widehat G;\Z_\epsilon)$. Orientifold discrete torsion instead has classes in $H^2(\widehat G;U(1)_\epsilon)$. The coefficient action is related, but the degree and geometric datum differ. The volume-line obstruction below is not identified with a gerbe, $B$-field, or Jandl obstruction.

\begin{proposition}\label{prop:phase-cocycle}
The cohomology class
\[
 [\mu_{\R}]\in H^2(\widehat G;\Z_\epsilon)
\]
is independent of the chosen real lifts.
\end{proposition}

\begin{proof}
Changing the lifts by an integer-valued $1$-cochain changes $\mu_\R$ by a twisted coboundary.
\end{proof}

The class $[\mu_\R]$ is not a single-sector odd age.  It measures the failure of the entire phase system to admit a coherent real lift.  This admits a precise cohomological formulation.

\begin{theorem}[Phase-lift obstruction]\label{thm:phase-obstruction}
With the standard inhomogeneous group-cohomology convention,
\begin{equation}\label{eq:bockstein-phase}
 \boxed{[\mu_{\R}]=-\beta([a])}
 \in H^2(\widehat G;\Z_\epsilon),
\end{equation}
where
\[
 \beta:H^1\bigl(\widehat G;(\R/\Z)_\epsilon\bigr)
 \longrightarrow H^2(\widehat G;\Z_\epsilon)
\]
is the connecting homomorphism associated with
\[
 0\longrightarrow\Z_\epsilon\longrightarrow\R_\epsilon
 \longrightarrow(\R/\Z)_\epsilon\longrightarrow0.
\]
The following conditions are equivalent:
\begin{enumerate}[label=\textup{(\roman*)}]
\item $[\mu_\R]=0$;
\item there are real lifts $\widehat a_\gamma$ satisfying
\begin{equation}\label{eq:coherent-phase-lift}
 \widehat a_{\gamma\delta}
 =\widehat a_\gamma+(-1)^{\epsilon(\gamma)}\widehat a_\delta;
\end{equation}
\item after multiplying $\Omega$ by a constant unit complex number, one can arrange
\begin{equation}\label{eq:phase-trivial-normalization}
 g^*\Omega=\Omega\quad(g\in G),
 \qquad
 \sigma^*\Omega=\overline\Omega\quad(\epsilon(\sigma)=1).
\end{equation}
\end{enumerate}
\end{theorem}

\begin{proof}
Let $\delta$ denote the inhomogeneous group-cohomology differential.  For the chosen real lift,
\[
 (\delta\widetilde a)(\gamma,\delta)
 =(-1)^{\epsilon(\gamma)}\widetilde a_\delta
 -\widetilde a_{\gamma\delta}+\widetilde a_\gamma
 =-\mu_\R(\gamma,\delta).
\]
This proves \eqref{eq:bockstein-phase}.  Thus $[\mu_\R]=0$ precisely when the lifts can be modified by an integer-valued cochain so that the defect vanishes, giving \eqref{eq:coherent-phase-lift}.

Assume a coherent lift exists.  Since $\widehat G$ is finite and $\R_\epsilon$ is a real vector-space module,
\[
 H^1(\widehat G;\R_\epsilon)=0;
\]
see, for example, \cite{Brown}.  An elementary averaging argument is useful here.  Put
\[
 S=\sum_{\gamma\in\widehat G}\widehat a_\gamma.
\]
For $\eta\in\widehat G$, summing
\[
 \widehat a_{\eta\gamma}
 =\widehat a_\eta+(-1)^{\epsilon(\eta)}\widehat a_\gamma
\]
over $\gamma$ gives
\[
 S=|\widehat G|\widehat a_\eta+(-1)^{\epsilon(\eta)}S.
\]
If $c=S/|\widehat G|$, then
\[
 \widehat a_\eta=c-(-1)^{\epsilon(\eta)}c.
\]
Hence $\widehat a_g=0$ for even $g$ and $\widehat a_\sigma=2c$ for odd $\sigma$.  Replacing $\Omega$ by $e^{-2\pi ic}\Omega$ leaves the even phases unchanged and subtracts $2c$ from every odd phase, which gives \eqref{eq:phase-trivial-normalization}.  The converse is immediate.
\end{proof}

\begin{proposition}[Classification of coherent lifts]\label{prop:phase-lift-classification}
Assume $[\mu_\R]=0$.  The set of coherent real lifts of the fixed phase cocycle is a torsor for
\[
 Z^1(\widehat G;\Z_\epsilon).
\]
Modulo integral twisted coboundaries, the set of equivalence classes is a torsor for
\[
 H^1(\widehat G;\Z_\epsilon).
\]
\end{proposition}

\begin{proof}
The difference of two coherent lifts is an integer-valued twisted $1$-cocycle, and adding such a cocycle to one lift produces another.  Quotienting by changes coming from integer-valued $0$-cochains gives the second statement.
\end{proof}

\begin{corollary}[Phase-trivial criterion]\label{cor:phase-trivial-criterion}
For a finite orientifold action preserving a Calabi--Yau volume line, the condition $[\mu_\R]=0$ is equivalent to the existence, after constant rephasing of the volume form, of the normalization \eqref{eq:phase-trivial-normalization}.
\end{corollary}

\begin{example}[A nontrivial obstruction]\label{ex:square-phase-obstruction}
Let $E_i=\C/(\Z+i\Z)$, $X=E_i\times E_i$ and $\Omega=dz_1\wedge dz_2$.  Define
\[
 r(z_1,z_2)=(iz_1,z_2),
 \qquad
 \tau(z_1,z_2)=(\overline z_1,\overline z_2).
\]
Then $\tau r\tau=r^{-1}$ and $r^*\Omega=i\Omega$.  Restricting the phase cocycle to $\langle r\rangle\cong C_4$ gives the character $r\mapsto1/4\in\R/\Z$.  Its Bockstein is a generator of
\[
 H^2(C_4;\Z)\cong\Z/4,
\]
so $[\mu_\R]\neq0$.  Thus this action cannot be made phase-trivial by rephasing $\Omega$.
\end{example}

Theorem~\ref{thm:phase-obstruction} is a criterion for normalization of the volume line, logically separate from the single-sector no-go theorem. In fact, for a finite graded group it is equivalent to triviality of the even volume character. Necessity is immediate from \eqref{eq:phase-trivial-normalization}. Conversely, if all even phases are zero, the cocycle law for a product of two odd elements gives $a_\sigma=a_\rho$ in $\R/\Z$. A single constant rephasing eliminates this common phase. This does not supply Euler class sewing in arbitrary dimension. The effective surface case is treated separately in Section~\ref{subsec:CY2-realization}.

\section{Admissible orientifold de Rham packages}\label{sec:admissible}

The previous section rules out determining the universal reflection-pair defect by individual conjugacy-invariant sector ages.  We now isolate sufficient finite-dimensional geometric data for a cohomological realization.  The definition is deliberately formulated in clean-intersection language and does not assume an analytic Fredholm realization.

\subsection{Sector data}
For every label $\gamma\in\widehat G$ specify a smooth sector $Z_\gamma$, a rank-one real local system $\mathcal L_\gamma$, and a locally constant degree shift $d_\gamma$. Even sectors are $Z_g=X^g$, with the componentwise rational shift
\begin{equation}\label{eq:even-shift}
 d_g|_F=2\age(g,F).
\end{equation}
For odd involutions we take $Z_\sigma=X^\sigma$ and require the chosen shifts $d_\sigma$ to be integers. These are additional data. If the graded group has noninvolutive odd elements, their sector data must also be supplied; the half-dimensional fixed-locus description is not asserted for them. All labels needed by multiplication must be included. In the surface application every odd element is an involution.

For an ordered interacting pair, specify a smooth correspondence
\begin{equation}\label{eq:pair-correspondence}
 Z_\gamma\xleftarrow{e_1}W_{\gamma,\delta}
 \xrightarrow{e_2}Z_\delta,\qquad
 m_{\gamma,\delta}:W_{\gamma,\delta}\longrightarrow Z_{\gamma\delta},
\end{equation}
where $m_{\gamma,\delta}$ is a proper embedding on each component. Write $c_{\gamma,\delta}$ for its real codimension. Let $K_{\gamma,\delta}\to W_{\gamma,\delta}$ be an actual real bundle of rank $k_{\gamma,\delta}$. On even pairs it is the underlying real Chen--Ruan obstruction bundle.

\paragraph{\textbf{Parity convention.}}
Throughout the formal realization theorem in this section, all correction ranks, Gysin embedding codimensions, and excess ranks in the composition diagrams are even. This is an explicit restriction, covering the purely complex and Calabi--Yau surface cases. We do not claim here an arbitrary odd-rank extension obtained just by suppressing its signs. The ordinary form degree, not the rational or shifted sector degree, controls wedge signs.

\begin{definition}[Euler-admissible orientifold de Rham package]\label{def:admissible-package}
An Euler-admissible package consists of the preceding data and parity convention, together with the following structures.
\begin{enumerate}[label=\textup{(A\arabic*)}]
\item \textbf{Degree defect.} On each connected interaction component, with shifts restricted to the relevant sector components,
\begin{equation}\label{eq:degree-defect}
 d_\gamma+d_\delta-d_{\gamma\delta}=k_{\gamma,\delta}+c_{\gamma,\delta}.
\end{equation}
\item \textbf{Coefficient-line compatibility.} Specify an isomorphism
\begin{equation}\label{eq:line-compatibility}
 e_1^*\mathcal L_\gamma\otimes e_2^*\mathcal L_\delta\otimes o(K_{\gamma,\delta})
 \cong m_{\gamma,\delta}^*\mathcal L_{\gamma\delta}\otimes o(N_{m_{\gamma,\delta}}).
\end{equation}
All identifications are ordered; $o(\cdot)$ denotes the orientation local system.
\item \textbf{Euler representatives.} Choose a closed form
$\mathbf e_{\gamma,\delta}\in\Omega^{k_{\gamma,\delta}}(W_{\gamma,\delta};o(K_{\gamma,\delta}))$
representing the actual twisted Euler class. Direct sums use the written order and the Whitney formula.
\item \textbf{Clean composition and Euler sewing.} For each triple, the two iterated correspondence fiber products are identified componentwise with a common smooth $W$, with the same three evaluation maps and the same proper final Gysin embedding $f:W\to Z_{\gamma_1\gamma_2\gamma_3}$. The base-change squares used in their compositions are clean, with actual excess bundles $E_L,E_R$ and their oriented clean base-change maps. Pull all correction bundles back to $W$ and put
\[
 B_L=K_{12}\oplus K_{12,3}\oplus E_L,\qquad
 B_R=K_{23}\oplus K_{1,23}\oplus E_R.
\]
They have the same rank $r$. Require an identification $\Phi:o(B_L)\to o(B_R)$ satisfying
\begin{equation}\label{eq:euler-sewing}
 \Phi_*e(B_L)=e(B_R)\quad\text{in }H^r(W;o(B_R)).
\end{equation}
The composed coefficient identifications from (A2) and from the oriented base-change squares agree under $\Phi$, and give the same coefficient system and orientation for $f_*$. These are requirements on the actual correspondence diagrams, not just on their virtual bundle classes.
\end{enumerate}
\end{definition}

An actual oriented isomorphism $B_L\cong B_R$, inducing the specified line and final Gysin identifications, is a sufficient stronger condition for \eqref{eq:euler-sewing}. A stable isomorphism is not sufficient. Equality of the Euler classes is the condition used below; the definition does not assume a canonical odd age or an analytic gluing theorem.

\subsection{The de Rham complex and product}
First form the unreduced complex
\[
 \widetilde{\AOri}^{\,q}=\bigoplus_{\gamma\in\widehat G}
 \Omega^{q-d_\gamma}(Z_\gamma;\mathcal L_\gamma),\qquad q\in\mathbb Q,
\]
where a summand is zero unless its ordinary form degree is a nonnegative integer. For conjugation-equivariant data, as specified in Section~\ref{subsec:descent}, set
\begin{equation}\label{eq:AOri}
 \AOri^q=(\widetilde{\AOri}^{\,q})^{\widehat G}
 \cong\bigoplus_{[\gamma]}
 \Omega^{q-d_\gamma}(Z_\gamma;\mathcal L_\gamma)^{C_{\widehat G}(\gamma)}.
\end{equation}
The exterior derivative is the unsuspended differential of ordinary local-coefficient forms and raises $q$ by one.

On unreduced homogeneous components, use chosen Thom representatives for
\begin{equation}\label{eq:admissible-product}
 \alpha_\gamma\star\beta_\delta
 =(m_{\gamma,\delta})_*
 \bigl(e_1^*\alpha_\gamma\wedge e_2^*\beta_\delta\wedge\mathbf e_{\gamma,\delta}\bigr),
\end{equation}
with the sum over interaction components understood. Here a tubular representative has the order $\pi^*\eta\wedge\Phi_N$, followed by extension by zero. The line isomorphism \eqref{eq:line-compatibility} is exactly the coefficient type required for this Gysin map. Products of collections use convolution,
\[
 (a\star b)_\eta=\sum_{\gamma\delta=\eta}a_\gamma\star b_\delta.
\]
The algebra asserted below is on cohomology, where the auxiliary differential-form choices disappear.

\begin{proposition}[Degree and de Rham descent]\label{prop:degree-leibniz}
The product \eqref{eq:admissible-product} has degree zero with respect to the shifted grading \eqref{eq:AOri}.  If $\alpha_\gamma$ has ordinary de Rham degree $p$, then on unsuspended differential forms
\begin{equation}\label{eq:ordinary-leibniz}
 d(\alpha_\gamma\star\beta_\delta)
 =d\alpha_\gamma\star\beta_\delta
 +(-1)^p\alpha_\gamma\star d\beta_\delta.
\end{equation}
Consequently the product of closed forms is closed and the product descends canonically to the shifted de Rham cohomology groups.
\end{proposition}

\begin{proof}
If the unshifted form degrees are $p$ and $q$, then the form under the pushforward has degree $p+q+k_{\gamma,\delta}$, and the Gysin map increases real degree by $c_{\gamma,\delta}$.  The output shifted degree is therefore
\[
 p+q+k_{\gamma,\delta}+c_{\gamma,\delta}+d_{\gamma\delta}
 =(p+d_\gamma)+(q+d_\delta)
\]
by \eqref{eq:degree-defect}.  Since the Euler representatives are closed, the ordinary de Rham Leibniz rule and the standard compatibility of Gysin pushforward with $d$ give \eqref{eq:ordinary-leibniz}.  Exact changes of either input therefore change the product by an exact form.
\end{proof}

\begin{remark}[Shifted grading and chain-level scope]\label{rem:suspension-sign}
The shifted degree need not have the same parity as the underlying form degree; it can even be rational. Formula~\eqref{eq:ordinary-leibniz} is an unsuspended formula for the chosen Thom representatives. Suspension signs alone do not supply coherent higher homotopies or a strict associative chain-level product. No such strictification is claimed here. The result is an associative algebra on shifted cohomology with its explicitly stated ordinary-form sign conventions.
\end{remark}

\subsection{Associativity from excess intersection}

The use of an Euler class on a double correspondence, and an excess comparison on a triple correspondence, is an established orbifold-product mechanism. A direct antecedent is the criterion recalled in \cite[Section~5]{LUXVirtual}, following Fantechi--G\"ottsche, in which the two correction products are compared after multiplication by their clean-excess Euler classes. For later de Rham and inertial formulations see \cite{KaufmannDeRham,EJKLog,EJKPlethora}. The hypotheses here record actual real bundles, local coefficients and a nonstable Euler comparison; the pull--push argument is not a new associativity principle. The virtual and rotated Chen--Ruan products use different correction classes, as made explicit in Section~\ref{subsec:virtual-comparison}.

We use the standard clean excess-intersection principle: composing two Gysin maps through two different clean fiber-product bracketings differs by multiplication with the Euler class of the corresponding excess bundle.  The appendix records the determinant-line convention used here.

\begin{theorem}[Associative realization]\label{thm:associative-realization}
An Euler-admissible package in Definition~\ref{def:admissible-package}, including its even-rank parity convention, defines an associative degree-zero product on unreduced cohomology. For conjugation-equivariant data it restricts to an associative product on $H^*(\AOri)$.
\end{theorem}
\begin{proof}
Fix closed representatives $\alpha,\beta,\gamma$ and a common triple component $W$. Apply the oriented clean base-change formula to the left composition, then the projection formula. Its cohomology class is the final pushforward of
\begin{equation}\label{eq:left-integrand}
 e_1^*\alpha\wedge e_2^*\beta\wedge e_3^*\gamma
 \wedge e(K_{12})\wedge e(K_{12,3})\wedge e(E_L).
\end{equation}
The right composition gives the final pushforward of
\begin{equation}\label{eq:right-integrand}
 e_1^*\alpha\wedge e_2^*\beta\wedge e_3^*\gamma
 \wedge e(K_{23})\wedge e(K_{1,23})\wedge e(E_R).
\end{equation}
Here the same symbol denotes a representative or its cohomology class as appropriate. All correction degrees and Gysin codimensions are even, so moving these factors to the indicated positions creates no sign. By Whitney multiplicativity their respective correction factors are $e(B_L)$ and $e(B_R)$. Hypothesis~\eqref{eq:euler-sewing} identifies these classes, and the last clause of (A4) identifies the coefficient systems and final pushforwards. Hence the two compositions agree. Summing over components and triples proves associativity on the group-labelled sum. Conjugation equivariance then gives the invariant statement.
\end{proof}

\begin{remark}[Stable data do not determine an Euler class]\label{rem:KO-not-enough}
The implication from an oriented stable isomorphism to equality of ordinary Euler classes is false. For example,
\[
 TS^2\oplus\underline\R\cong\underline\R^2\oplus\underline\R,
 \qquad e(TS^2)=2u\ne0=e(\underline\R^2).
\]
Both sides are stably isomorphic even as oriented bundles; choosing Euler forms separately cannot change these classes. Multiplication by the Euler class of the stabilizing trivial line gives only $0=0$. Thus actual Euler/Thom representatives and stable sewing do not replace \eqref{eq:euler-sewing}. An actual compatible isomorphism of $B_L$ and $B_R$ does suffice. A generalized cohomology Euler class would define a different construction requiring its own hypotheses.
\end{remark}

\paragraph{\textbf{Comparison with complex inertial bundles.}}
In the complex setting of \cite[Proposition~3.1.4]{EJKPlethora}, the relevant $K$-class identity determines the total Chern classes of the two equal-rank total bundles, hence their top Chern classes. This does not justify cancelling a stabilizing \emph{real} bundle in the ordinary Euler class. Our explicit hypothesis \eqref{eq:euler-sewing} concerns precisely that distinction.

\subsection{Basic sources of admissible data}

Definition~\ref{def:admissible-package} is designed to separate existence questions from the formal de Rham argument.  We record several situations that motivate its clauses.

\paragraph{Purely even sectors.}
For the ordinary Chen--Ruan theory of a finite abelian global quotient, take $Z_g=X^g$, $d_g=2\age(g)$ and $\mathcal L_g=\underline\R$.  Let $K_{g,h}$ be the underlying real bundle of the complex obstruction bundle.  The complex orientations provide the coefficient-line data, and the standard obstruction/excess identity supplies sewing.  Thus the Chen--Hu model is an admissible package.

\paragraph{\textbf{Transverse reflection pairs.}}
Suppose two odd fixed loci $L_\sigma$ and $L_\rho$ meet transversely and $\sigma\rho$ is even.  Then $W_{\sigma,\rho}=L_\sigma\cap L_\rho$ maps naturally to $X^{\sigma\rho}$.  Even when the geometric obstruction bundle is zero, the Gysin map requires the orientation line $o(N_m)$ and the grading defect depends on the real codimension.  If a relative grading and compatible orientations of the pair are chosen, this gives the simplest nontrivial odd--odd piece of an admissible package.  Changing the relative grading can change the degree assignment even though neither odd fixed locus changes as a submanifold.

\paragraph{\textbf{Possible analytic sources.}}
Families of Real Cauchy--Riemann operators provide one possible analytic source of relative gradings and orientation lines.  When such analytic data admit finite-dimensional Euler/Thom representatives compatible with clean composition, they can feed into Definition~\ref{def:admissible-package}.  The construction of the corresponding real family-index sewing identity is a separate analytic problem and is not part of the proof of Theorem~\ref{thm:associative-realization} or of the Calabi--Yau surface realization below.

\paragraph{\textbf{Coefficients in characteristic two.}}
One may formulate a singular cohomology analogue over $\Ftwo$, where orientation signs and orientation local systems disappear. This is not a de Rham theory over $\Ftwo$. Degree compatibility and the appropriate Euler/excess comparison remain necessary; changing coefficients does not by itself solve the geometric sewing problem. Returning to integral or real coefficients requires the orientation-line coherence omitted by the mod-two theory.

These examples also show why the admissibility assumptions should not be compressed into the phrase ``the Real obstruction bundle is oriented.''  The coefficient line, degree defect and four-point sewing are separate pieces of structure.

\subsection{Descent, units and independence of auxiliary forms}\label{subsec:descent}

The unreduced formulas above are written with group elements as labels.  For a quotient stack one takes the direct sum over conjugacy classes and the invariants of the corresponding centralizers.  We require the data of an admissible package to be equivariant under conjugation: if $k\in\widehat G$, the package for $(\gamma,\delta)$ is carried isomorphically to the package for $(k\gamma k^{-1},k\delta k^{-1})$.  Under this hypothesis convolution preserves the invariant subspace in \eqref{eq:AOri}. The identification with the centralizer sum uses orbit sums of equivariantly transported classes. This convention fixes all multiplicities; no unspecified centralizer scalar is inserted into the multiplication.

The identity sector is required to be normalized by
\[
 Z_1=X,\qquad d_1=0,\qquad \mathcal L_1=\underline\R,
\]
and by taking $W_{1,\gamma}=W_{\gamma,1}=Z_\gamma$, $K_{1,\gamma}=K_{\gamma,1}=0$ with the evident line identifications.  Then the constant function $1\in H^0(X;\R)$ is a two-sided unit.

\begin{proposition}[Independence of Euler representatives]\label{prop:euler-choice}
The product induced on cohomology is independent of the choice of closed Euler forms $\mathbf e_{\gamma,\delta}$ representing the fixed Euler classes of the bundles $K_{\gamma,\delta}$.
\end{proposition}

\begin{proof}
If $\mathbf e'$ is another representative, then $\mathbf e'-\mathbf e=d\xi$ for a twisted form $\xi$.  Replacing $\mathbf e$ by $\mathbf e'$ in \eqref{eq:admissible-product} changes the integrand by an exact form plus terms which vanish for closed inputs.  Since the Gysin map commutes with $d$, the difference of the two products is exact.  Hence the cohomology class is unchanged.
\end{proof}

\begin{proposition}[Unit]\label{prop:unit}
With the normalization above, the identity-sector class $1\in H^0(X;\R)$ is the unit of the cohomology algebra determined by an admissible package.
\end{proposition}

\begin{proof}
For $(1,\gamma)$ and $(\gamma,1)$ the correspondence is the identity on $Z_\gamma$, the obstruction bundle is zero and the Euler form is $1$.  Formula \eqref{eq:admissible-product} is therefore the ordinary wedge product with the constant function $1$ followed by the identity Gysin map.
\end{proof}

The product need not be written as a strictly associative multiplication on a particular choice of de Rham representatives.  The theorem asserts associativity on cohomology.  A canonical chain level $A_\infty$ refinement would require coherent choices over higher sewing spaces and is not part of the present construction.

\subsection{Recovery of Chen--Hu}

\begin{corollary}[Even reduction]\label{cor:ChenHu-reduction}
Suppose all labels are even, set $d_g=2\age(g)$, take $\mathcal L_g$ trivial, and let $K_{g,h}$ be the underlying real bundle of the Chen--Ruan obstruction bundle.  Then \eqref{eq:admissible-product} is the Chen--Hu product.
\end{corollary}

\begin{proof}
The real rank of $K_{g,h}$ is twice the complex obstruction rank and the real codimension of $W\subset X^{gh}$ is twice its complex codimension.  Thus \eqref{eq:degree-defect} is exactly twice the age identity \eqref{eq:age-rank}.  Complex orientations trivialize all orientation local systems, and the real Euler class of the underlying oriented real bundle is the ordinary top Chern/Euler class.  The formula reduces to \eqref{eq:CH-product}.
\end{proof}

The preceding reduction is a comparison with the established even-sector theory, not a new construction of that theory. Its formal Thom formulation for general orbifolds is given by Hu--Wang \cite[Theorem~3.8]{HuWang}. Their complex obstruction class identity \cite[Theorem~3.2]{HuWang} does not remove the ordinary real Euler class requirement \eqref{eq:euler-sewing}; in particular it does not justify stable real cancellation.

\subsection{Phase-trivial Calabi--Yau surface realizations}\label{subsec:CY2-realization}

We now identify the proposed surface product with an ordinary Chen--Ruan product after an invariant rotation of almost complex structure.  Let $(X,\omega,J,\Omega)$ be a compact connected K\"ahler surface with nowhere vanishing holomorphic two-form.  Let $G$ be a finite abelian group acting effectively by holomorphic symplectomorphisms and assume
\begin{equation}\label{eq:CY2-even-volume}
 g^*\Omega=\Omega,
 \qquad g\in G.
\end{equation}
Let $\tau$ be an anti-holomorphic anti-symplectic involution satisfying
\begin{equation}\label{eq:CY2-odd-volume}
 \tau^*\Omega=\overline\Omega,
 \qquad
 \tau g\tau=g^{-1}.
\end{equation}
Set $\widehat G=G\rtimes\langle\tau\rangle$.  Since $G$ is abelian, every odd element $g\tau$ is an involution.

\begin{lemma}[Even fixed loci]\label{lem:CY2-even-fixed}
Let $1\neq g\in G$.  Then $X^g$ is empty or zero-dimensional.  At every fixed point,
\[
 \age(g)=1.
\]
\end{lemma}

\begin{proof}
At $p\in X^g$, equation \eqref{eq:CY2-even-volume} implies $\det_\C(dg_p)=1$.  If one eigenvalue were $1$, then both would be $1$.  A finite-order holomorphic action is locally linearizable, so $dg_p=1$ would make $g$ the identity on a neighborhood of $p$, hence on connected $X$, a contradiction.  Thus the eigenvalues are
\[
 e^{2\pi i\theta},\qquad e^{2\pi i(1-\theta)},
 \qquad 0<\theta<1,
\]
and the age is $1$.
\end{proof}

\begin{lemma}[Canonical odd orientation and grading]\label{lem:CY2-odd-orientation}
Let $\sigma\in\widehat G\setminus G$ and let $L_\sigma=X^\sigma$.  Every nonempty component of $L_\sigma$ is an oriented Lagrangian surface.  Its squared phase with respect to $\Omega$ is identically one, and hence it has vanishing Maslov class and the canonical zero grading in the sense of \cite{Seidel}.
\end{lemma}

\begin{proof}
The fixed locus of the anti-symplectic involution $\sigma$ is Lagrangian.  Since every odd element has the form $g\tau$, equations \eqref{eq:CY2-even-volume}--\eqref{eq:CY2-odd-volume} give $\sigma^*\Omega=\overline\Omega$.  For $v_1,v_2\in T_xL_\sigma$,
\[
 \Omega(v_1,v_2)=\overline{\Omega(v_1,v_2)}.
\]
On a Lagrangian plane in complex dimension two the restriction of the nonzero $(2,0)$-form is nonzero, hence real of constant sign on each connected component.  We orient the component by requiring $\operatorname{Re}\Omega|_{L_\sigma}>0$.  The resulting phase is zero and its squared phase is constant one.  The final assertion is the standard phase-lift characterization of graded Lagrangians \cite{Seidel}.
\end{proof}

\begin{remark}[Dependence on the normalized volume form]\label{rem:Omega-sign}
The orientation in Lemma~\ref{lem:CY2-odd-orientation} is canonical relative to the chosen phase-trivial volume form $\Omega$.  Multiplying $\Omega$ by a positive real constant changes neither the odd orientations nor the grading.  Replacing $\Omega$ by $-\Omega$ reverses the orientations of all odd fixed surfaces while leaving their squared-phase grading unchanged.  Throughout the surface realization theorem the normalized form $\Omega$ is regarded as part of the Calabi--Yau orientifold data.
\end{remark}

Define sector shifts by
\begin{equation}\label{eq:CY2-shifts}
 d_1=0,
 \qquad
 d_g=2\quad(1\neq g\in G),
 \qquad
 d_\sigma=1\quad(\epsilon(\sigma)=1).
\end{equation}
All coefficient local systems are trivialized by the complex orientation on even sectors and by the $\operatorname{Re}\Omega$ orientation on odd fixed surfaces.

For an interaction pair set
\[
 W_{\gamma,\delta}=Z_\gamma\cap Z_\delta,
 \qquad
 m_{\gamma,\delta}:W_{\gamma,\delta}\hookrightarrow Z_{\gamma\delta}.
\]
For even-even pairs take the ordinary Chen--Ruan obstruction bundle.  Whenever at least one label is odd, set
\begin{equation}\label{eq:CY2-Kzero}
 K_{\gamma,\delta}=0.
\end{equation}

\begin{lemma}[Surface degree defect]\label{lem:CY2-degree-defect}
The data \eqref{eq:CY2-shifts}--\eqref{eq:CY2-Kzero} satisfy the degree-defect identity \eqref{eq:degree-defect} on every nonempty interaction component.
\end{lemma}

\begin{proof}
The even-even case is twice the ordinary Chen--Ruan age identity.  For an odd involution $\sigma$, the self-interaction $L_\sigma\hookrightarrow X$ has real codimension two and
\[
 d_\sigma+d_\sigma-d_1=2.
\]
If $\sigma\neq\rho$ are odd, then $g=\sigma\rho\neq1$ is even.  The intersection $L_\sigma\cap L_\rho$ is contained in the discrete set $X^g$, so its inclusion into $Z_g=X^g$ has codimension zero and
\[
 d_\sigma+d_\rho-d_g=1+1-2=0.
\]
Finally, if $1\neq g\in G$ and $\sigma$ is odd, the interaction locus is discrete whereas $Z_{g\sigma}=L_{g\sigma}$ is a surface.  Thus the real codimension is two and
\[
 d_g+d_\sigma-d_{g\sigma}=2+1-1=2.
\]
The reversed mixed product is identical.
\end{proof}

The binary products involving an odd sector therefore require no extra obstruction Euler factor.  They are the canonically oriented clean-intersection Gysin correspondences
\begin{equation}\label{eq:CY2-mixed-product}
 \alpha_\gamma\star\beta_\delta
 =(m_{\gamma,\delta})_*
 \bigl(e_1^*\alpha_\gamma\wedge e_2^*\beta_\delta\bigr)
\end{equation}
whenever at least one label is odd.

\begin{lemma}[Surface sewing by invariant rotation]\label{lem:CY2-sewing}
Under the hypotheses of this subsection there is a $\widehat G$-invariant almost complex structure $J_\Omega$ for which the products consisting of the original even Chen--Hu product and \eqref{eq:CY2-mixed-product} are precisely the ordinary global-quotient Chen--Ruan products. In particular, their Euler/excess comparisons and associativity hold on cohomology.
\end{lemma}
\begin{proof}
We give the construction and the product comparison, including the orientation and rank checks.

\emph{Construction of the rotation.}
Let $h(u,v)=\omega(u,Jv)$ and define the skew-adjoint real endomorphism $A$ by
\[
 h(Au,v)=\operatorname{Re}\Omega(u,v).
\]
At each point choose a $J$-unitary frame in which $\Omega=c\,dz_1\wedge dz_2$ with $c>0$. Direct real-linear computation gives
\[
 A^2=-c^2\operatorname{Id},\qquad AJ=-JA,\qquad
 J_\Omega:=c^{-1}A,\qquad
 J_\Omega(z_1,z_2)=(-\overline z_2,\overline z_1).
\]
The positive scale $c$ is intrinsically characterized by $A^2=-c^2\operatorname{Id}$, so these definitions are smooth and independent of the frame. Thus $J_\Omega^2=-1$. The path
\[
 J_t=(\cos t)J+(\sin t)J_\Omega,\qquad 0\leq t\leq\pi/2,
\]
is a path of almost complex structures and preserves the ambient orientation. No integrability or Ricci-flatness assertion is required.

\emph{Equivariance and fixed-locus orientations.}
An even element preserves $h$ and $\operatorname{Re}\Omega$. An odd element reverses both $\omega$ and $J$, so it also preserves $h$; by phase triviality it preserves $\operatorname{Re}\Omega$. Consequently every element of $\widehat G$ commutes with $A$ and preserves $c$, hence with $J_\Omega$. Even elements commute with the whole path $J_t$.

For an odd involution $\sigma$, its fixed tangent space is the $+1$ eigenspace of $d\sigma$, and is a complex line for $J_\Omega$. The normal is its $-1$ complex eigenline. For nonzero tangent $v$,
\[
 \operatorname{Re}\Omega(v,J_\Omega v)=c\,h(J_\Omega v,J_\Omega v)>0.
\]
Hence the $J_\Omega$ orientation of $L_\sigma$ equals the specified $\operatorname{Re}\Omega$ orientation. The normal orientation induced by $J_\Omega$ agrees with the ambient/sector Gysin convention. Distinct reflection surfaces meeting at a point are transverse $J_\Omega$-complex lines, so their local intersection sign is $+1$.

\emph{Ages and obstruction ranks.}
The even-equivariant path $J_t$ keeps the finite-order complex eigenvalue multiplicities constant on each even fixed component. By Lemma~\ref{lem:CY2-even-fixed}, every nontrivial even fixed sector has $J_\Omega$-age $1$. The odd tangent/normal eigenvalues are $+1,-1$, giving age $1/2$. Thus the shifts are exactly those of \eqref{eq:CY2-shifts}.

For the ordinary orbifold product of the $J_\Omega$-almost complex quotient, the obstruction rank on a common component $W$ is
\[
 r_{\gamma,\delta}
 =\age_{J_\Omega}(\gamma)+\age_{J_\Omega}(\delta)
  -\age_{J_\Omega}(\gamma\delta)
  -\operatorname{codim}_{\C}(W,X^{\gamma\delta}).
\]
This rank formula holds for noncommuting pairs as well: it follows from the representation-theoretic obstruction formula of \cite[Theorem~2]{Hepworth}, equivalently from the orbifold Riemann--Roch calculation in \cite{ChenRuan}. The intrinsic complex obstruction-class identity of Hu--Wang \cite[Theorem~3.2]{HuWang} gives a further formulation of this ordinary orbifold input. No simultaneous eigenline decomposition for $\widehat G$ is assumed. The complete list of nonempty pairs involving an odd label is
\[
\begin{array}{c|c|c|c}
\text{pair type}&\text{input age sum}&\text{output age}&\text{complex codimension}\\ \hline
(1,\sigma),\ (\sigma,1)&1/2&1/2&0\\
(\sigma,\sigma)&1&0&1\\
(\sigma,\rho),\ \sigma\ne\rho&1&1&0\\
(g,\sigma),\ (\sigma,g),\ g\ne1&3/2&1/2&1
\end{array}
\]
Every row has rank zero. A rank-zero obstruction bundle is the zero bundle, so the ordinary orbifold product is exactly the oriented Gysin expression \eqref{eq:CY2-mixed-product}.

\emph{Comparison on even pairs.}
Pairs with an identity label have the ordinary restriction product. For $g\ne1$ and $h=g^{-1}$, the common fixed set is discrete, the output is $X$, the complex codimension is two, and the obstruction rank is zero. The Gysin orientations agree because the ambient complex structures are homotopic. If $g,h,gh\ne1$, the common fixed set is discrete and the obstruction rank is one; its Euler class vanishes on that set. These cases exhaust even pairs and agree with the original $J$-Chen--Hu product.

All binary correspondences, shifts, correction classes and orientations have therefore been identified with those of the ordinary almost complex orbifold. The Chen--Ruan associativity theorem \cite{ChenRuan}, in its group-labelled global quotient formulation \cite{FantechiGottsche}, proves the claim. In particular, its genuine obstruction/excess comparison supplies the Euler identity required here; an explicit vector-bundle formulation is \cite[Theorem~15 and equation~(5)]{Hepworth}. This argument does not cancel a stabilizing real bundle.
\end{proof}

The rotation principle has precedents in hyper-K\"ahler geometry; see \cite[Section~3]{YoshikawaRealK3} for Donaldson's trick on K3 surfaces and Biswas--Wilkin \cite{BiswasWilkin}. The pointwise construction above is stated explicitly because the given K\"ahler metric need not be Ricci-flat. The proof only uses the almost complex theory; it does not infer integrability for an arbitrary given metric from the integrable hyper-K\"ahler case.

For clarity about another orientation issue, the original map $J:TL_\sigma\to N_{L_\sigma/X}$ is a real bundle isomorphism, but reverses orientation when the normal is oriented using the ambient complex orientation and $\operatorname{Re}\Omega|_{L_\sigma}>0$. The permutation from $(v_1,Jv_1,v_2,Jv_2)$ to $(v_1,v_2,Jv_1,Jv_2)$ has sign $-1$. Hence
\[
 e(N_{L_\sigma/X})=-e(TL_\sigma),\qquad
 [L_\sigma]^2=-\chi(L_\sigma)
\]
for each compact component. This is consistent with Proposition~\ref{prop:normal-JTL}, which asserts a real bundle and orientation-local-system isomorphism, not preservation of these chosen orientations.

\begin{theorem}[Calabi--Yau surface realization]\label{thm:CY2-realization}
Under the effective-action hypotheses \eqref{eq:CY2-even-volume}--\eqref{eq:CY2-odd-volume}, let
\begin{equation}\label{eq:CY2-state-space}
 \widetilde H^q_{\Ori}(X,\widehat G)
 =H^q(X;\R)\oplus\bigoplus_{1\ne g\in G}H^{q-2}(X^g;\R)
 \oplus\bigoplus_{\epsilon(\sigma)=1}H^{q-1}(L_\sigma;\R).
\end{equation}
Its group-labelled product is associative; the even restriction is the original Chen--Hu product and all pairs involving an odd label use \eqref{eq:CY2-mixed-product}. Taking conjugation invariants with orbit-sum normalization gives a unital graded algebra and an identification
\[
 H^*_{\Ori}(X,\widehat G)
 :=(\widetilde H^*_{\Ori}(X,\widehat G))^{\widehat G}
 \cong H^*_{\CR}([(X,J_\Omega)/\widehat G];\R).
\]
The construction is canonical relative to $(\omega,J,\Omega)$; no extra odd obstruction bundle is required.
\end{theorem}
\begin{proof}
Lemma~\ref{lem:CY2-sewing} identifies the entire product with the ordinary group-labelled orbifold product and identifies all orientations and shifts. It also proves equivariance and associativity. The identity-sector constant is a two-sided unit. Passing to invariants is the standard orbit-sum/centralizer identification.
\end{proof}

Effectivity is needed in Lemma~\ref{lem:CY2-even-fixed}: a nontrivial kernel element has fixed locus $X$ and age zero, not an isolated fixed set of age one. Quotienting out a kernel changes the inertia. An ineffective extension must instead retain its additional sectors and recompute their shifts and products. Also, the zero lift of the squared Lagrangian phase in Lemma~\ref{lem:CY2-odd-orientation} and the shift $d_\sigma=1$ are different notions: the latter is twice the rotated orbifold age.

\begin{corollary}[Frobenius property]\label{cor:CY2-Frobenius}
If $X$ is compact, the trace supported on the identity sector,
\[
 \operatorname{Tr}(\alpha)=\int_X\alpha,
\]
defines a nondegenerate invariant pairing
\[
 \langle\alpha,\beta\rangle=\operatorname{Tr}(\alpha\star\beta).
\]
Thus the Calabi--Yau surface orientifold algebra is a finite-dimensional Frobenius algebra, not necessarily commutative.
\end{corollary}

\begin{proof}
The identity sector is paired by ordinary Poincar\'e duality.  Nontrivial even sectors pair with their inverse point sectors.  An odd involution pairs with itself, and the trace of the self-Gysin product is
\[
 \int_{L_\sigma}\alpha\wedge\beta.
\]
Poincar\'e duality on each canonically oriented compact surface $L_\sigma$ is nondegenerate. Every element of $\widehat G$ preserves the ambient orientation, since the complex dimension is two, so the trace is $\widehat G$-invariant. For a nonzero invariant vector $v$, choose $w$ with $\langle v,w\rangle\ne0$ and average $w$ over the group; invariance of the pairing preserves that value. Thus restriction to finite-group invariants is nondegenerate. Associativity gives invariance of the trace pairing.
\end{proof}

By Corollary~\ref{cor:phase-trivial-criterion}, the phase-trivial hypothesis in this subsection is intrinsic: for any finite Calabi--Yau orientifold action preserving the volume line, it can be achieved after rephasing $\Omega$ exactly when $[\mu_\R]=0$.

\section{Dihedral and real surface examples}\label{sec:examples}

The first examples make the closed-sector involution explicit and exhibit the grading constraints on odd interactions. The Eisenstein calculation supplies a full algebra with nontrivial even point sectors. The subsequent square-torus and real K3 examples vary the topology of the reflection locus, test nonzero normal Euler classes, and distinguish algebras with identical graded dimensions.

\subsection{The local reflection pair on the complex line}

Consider the two anti-linear involutions
\[
 \sigma_0(z)=\overline z,
 \qquad
 \sigma_\theta(z)=e^{2\pi i\theta}\overline z.
\]
Their fixed loci are the real lines
\[
 L_0=\R,
 \qquad
 L_\theta=e^{\pi i\theta}\R.
\]
As proved in Lemma~\ref{lem:unitary-conjugacy}, the individual involutions are unitarily equivalent.  Their product, however, is the rotation
\[
 \sigma_0\sigma_\theta=e^{-2\pi i\theta}.
\]
Thus the relative angle is visible in the even output sector although it is invisible in either odd sector considered separately.

If $\theta\notin\Z$, the two real lines are transverse at the origin.  Relative angles enter Fredholm indices only after a domain, boundary gradings, asymptotic conditions and a Fredholm setup have been specified. The pair of lines by itself does not construct the bundle $K_{\sigma_0,\sigma_\theta}$ or its coherent sewing maps. This example illustrates the relative geometric datum and the failure of a universal single-sector age; it is not a construction of an admissible package.

\subsection{The order eight dihedral action on the projective line}

Let
\begin{equation}\label{eq:D4}
 D_8=\langle r,s\mid r^4=s^2=1,\ srs=r^{-1}\rangle,
\end{equation}
where $D_8$ denotes the dihedral group of order eight.  Let
\[
 r[z_0:z_1]=[z_0:iz_1],
 \qquad
 s[z_0:z_1]=[\overline z_0:\overline z_1].
\]
The even subgroup is $G=\langle r\rangle\cong\Z_4$ and $s$ is anti-holomorphic.  The two torus fixed points are
\[
 p_0=[1:0],\qquad p_\infty=[0:1].
\]
For $r$, the tangent weights at $p_0$ and $p_\infty$ are respectively $1/4$ and $3/4$.  Thus
\begin{center}
\begin{tabular}{c|cc}
sector & $p_0$ & $p_\infty$\\ \hline
$r$   & $1/4$ & $3/4$\\
$r^2$ & $1/2$ & $1/2$\\
$r^3$ & $3/4$ & $1/4$
\end{tabular}
\end{center}
The even Chen--Ruan sector space is therefore the identity cohomology of $\mathbb{CP}^1$ together with six point-sector generators carrying the indicated shifts.

For the unreduced global-quotient product, the local multiplication can be read directly from the weights.  Let $u_{k,0}$ and $u_{k,\infty}$ denote the unit classes of the $r^k$ point sectors.  At $p_0$, the products
\[
 u_{1,0}\star u_{1,0}=u_{2,0},
 \qquad
 u_{1,0}\star u_{2,0}=u_{3,0}
\]
have no obstruction contribution.  The product $u_{1,0}\star u_{3,0}$ lands in the identity sector; its local intersection contribution is the point class $[p_0]\in H^2(\mathbb{CP}^1)$.  At $p_\infty$, the corresponding weights are reversed.  In particular, the product $u_{1,\infty}\star u_{1,\infty}$ has a rank-one obstruction bundle over a point and therefore vanishes in ordinary cohomology for degree reasons, whereas $u_{3,\infty}\star u_{3,\infty}=u_{2,\infty}$.  These asymmetries are exchanged by inversion of the rotation label and are exactly accounted for by the Chen--Ruan degree shifts.

These formulas use the unreduced convention followed by the orbit-sum invariant identification. Quotient stack integration changes the trace normalization; no additional unspecified scalar is inserted into this multiplication table. The point of the example is the even-sector geometry and reflection action.

The reflection $s$ normalizes $G$ by $srs=r^{-1}$.  For the closed-sector involution, however,
\[
 q(r)=sr^{-1}s=r,
\]
so $\RR$ preserves the $r$-label and acts anti-linearly on each fixed-point component with the phase $e^{\pi i\age(r,p)}$.  The same holds for $r^2$ and $r^3$.  Theorem~\ref{thm:closed-real} gives a real form of the full even Chen--Ruan algebra.

Now consider the odd elements
\[
 s_j=r^j s,
 \qquad j=0,1,2,3.
\]
Each $L_j=X^{s_j}$ is a circle in $\mathbb{CP}^1\cong S^2$, and every $L_j$ passes through $p_0$ and $p_\infty$.  The group law gives
\begin{equation}\label{eq:odd-odd-D4}
 s_i s_j=r^{i-j}.
\end{equation}
Thus a product of two reflection states, when defined, must land in the rotation sector labeled by $r^{i-j}$.  For $i\neq j$, the interaction locus is supported at the two points $p_0,p_\infty$.  The relative tangent lines of $L_i$ and $L_j$ at those points determine the Real Fredholm orientation signs.

Equation \eqref{eq:odd-odd-D4} is important conceptually.  The output age depends on $i-j$, while the individual reflection germs are all locally conjugate.  Hence the degree defect cannot be recovered from two single-sector odd ages.  A putative extension must satisfy the degree and coefficient constraints, not merely choose orientation signs. In particular this example does not admit the integer odd shifts of Definition~\ref{def:admissible-package} together with all these even shifts: at $p_0$ a reflection pair with product $r$ would require
\[
 d_{s_1}+d_{s_0}-\tfrac12=k_{s_1,s_0}+c_{s_1,s_0}\in\Z,
\]
which is impossible for integer $d_{s_i}$. In addition, a reflection self-Gysin map on $\mathbb{CP}^1$ has odd real codimension, outside our even-codimension formal theorem. Thus this is an illustration of even-sector and relative fixed-locus geometry, not an odd-sector realization; determinant data alone cannot remove its grading obstruction.

\subsection{An order six dihedral action on the projective plane}

Let
\[
 D_6=\langle r,s\mid r^3=s^2=1,\ srs=r^{-1}\rangle
\]
act by
\begin{equation}\label{eq:D6-action}
 r[z_0:z_1:z_2]=[z_0:\omega z_1:z_2],
 \qquad
 s[z_0:z_1:z_2]=[\overline z_0:\overline z_1:\overline z_2],
\end{equation}
where $\omega=e^{2\pi i/3}$.  Then $srs=r^{-1}$ because complex conjugation sends $\omega$ to $\omega^{-1}$.

The rotation $r$ has two connected fixed components: the hyperplane
\[
 H=\{z_1=0\}\cong\mathbb{CP}^1
\]
and the point $p_1=[0:1:0]$.  Along $H$ the normal line is the $z_1$ direction and has $r$-weight $1/3$.  Hence
\[
 \age(r,H)=\frac13,
 \qquad
 \age(r^2,H)=\frac23.
\]
For the product of two $r^2$ sectors along $H$, the three-point normal weights are
\[
 \frac23+\frac23+\frac23=2.
\]
Thus the ordinary Chen--Ruan obstruction bundle has complex rank one.  It is the normal line
\[
 N_{H/\mathbb{CP}^2}\cong\mathcal O_{\mathbb{CP}^1}(1),
\]
so its Euler class is the nonzero class
\[
 c_1\bigl(\mathcal O(1)\bigr)\in H^2(H;\Z).
\]
This gives a nontrivial positive-dimensional Chen--Hu twist factor on the even part.

The odd reflection $s$ is standard complex conjugation, so its fixed locus is
\[
 X^s=\mathbb{RP}^2.
\]
The other two reflection loci are the fixed sets of $rs$ and $r^2s$ and are again real projective planes obtained by rotating $\mathbb{RP}^2$.  Their pairwise products are labeled by $r$ or $r^2$.  As before, the labels and the clean intersections are explicit, but a cohomological odd--odd product additionally requires determinant/Fredholm orientation data.  This example displays a nonzero even obstruction Euler class together with a positive-dimensional Real reflection locus, but does not realize the full integer-odd-shift package. Along the hyperplane the output $r$-sector has shift $2/3$; for a pair of reflections producing $r$, the difference of integer odd shifts and $2/3$ cannot equal an integer rank plus codimension. The nonorientability of $\mathbb{RP}^2$ is a further coefficient issue, separate from this grading obstruction.

\subsection{What toric data determine}

For a smooth toric global quotient, the ordinary even Chen--Hu product is algorithmic: at a torus fixed component one reads the fractional weights of two group elements, determines the obstruction eigenbundles, and multiplies by their Euler classes.  Theorem~\ref{thm:closed-real} then adds an orientifold real structure on this even algebra whenever an anti-complex involution normalizes the even group.

For odd sectors, toric coordinates can describe the relative fixed-locus geometry, but they do not produce a canonical single-sector age.  To extend the toric algorithm across reflection sectors one must supplement the fan/weight data by an admissible determinant/Fredholm package.  In this sense the present framework separates the combinatorial part of the problem from the Real orientation problem.

\subsection{The Eisenstein abelian surface}\label{subsec:eisenstein}

We finish with an effective compact example whose surface product is canonically defined from the normalized Calabi--Yau data, identified with the rotated ordinary orbifold ring by Theorem~\ref{thm:CY2-realization}, and computed explicitly.  Let
\[
 \omega=e^{2\pi i/3},
 \qquad
 E_\omega=\C/(\Z+\omega\Z),
 \qquad
 X=E_\omega\times E_\omega,
\]
with holomorphic volume form
\[
 \Omega=dz_1\wedge dz_2.
\]
Define
\begin{equation}\label{eq:eisenstein-action}
 r(z_1,z_2)=(\omega z_1,\omega^{-1}z_2),
 \qquad
 \tau(z_1,z_2)=(-\overline z_1,-\overline z_2).
\end{equation}
Then
\[
 r^3=\tau^2=1,
 \qquad
 \tau r\tau=r^{-1},
\]
so $\widehat G=\langle r,\tau\rangle$ is the order-six dihedral group, isomorphic to $S_3$.  Moreover
\[
 r^*\Omega=\Omega,
 \qquad
 \tau^*\Omega=\overline\Omega,
\]
so Theorem~\ref{thm:CY2-realization} applies.

\subsubsection{The rotation fixed set}

Put
\[
 \xi=\frac{2+\omega}{3}.
\]
Since $(1-\omega)\xi=1$ in the lattice,
\[
 K=\{0,\xi,2\xi\}=\operatorname{Fix}(z\mapsto\omega z)
\]
is the three-point fixed set on $E_\omega$.  The same set is fixed by $z\mapsto-\overline z$.  Hence
\begin{equation}\label{eq:eisenstein-F}
 F=X^r=X^{r^2}=K\times K,
 \qquad |F|=9.
\end{equation}
For $p\in F$, let
\[
 u_{1,p}\in H^0(X^r),
 \qquad
 u_{2,p}\in H^0(X^{r^2})
\]
denote the point-sector unit classes.  Both have shifted degree two.

The even Chen--Ruan products are
\begin{equation}\label{eq:eisenstein-even-point-products}
 u_{1,p}\star u_{1,p'}=0,
 \qquad
 u_{2,p}\star u_{2,p'}=0,
 \qquad
 u_{1,p}\star u_{2,p'}=\delta_{p,p'}q,
\end{equation}
where $q\in H^4(X)$ is the normalized point class.  The first two products vanish because the corresponding Chen--Ruan obstruction bundle has positive rank over a point.

\subsubsection{The three reflection tori}

Let
\[
 s_i=r^i\tau,
 \qquad
 L_i=X^{s_i},
 \qquad i=0,1,2.
\]
Write $z=x+y\omega$ on one Eisenstein factor.  The relevant real circles have equations
\[
 C_0:\ 2x-y=0\pmod1,
 \qquad
 C_1:\ x+y=0\pmod1,
 \qquad
 C_2:\ 2y-x=0\pmod1.
\]
Then
\begin{equation}\label{eq:eisenstein-Li}
 L_0=C_0\times C_0,
 \qquad
 L_1=C_1\times C_2,
 \qquad
 L_2=C_2\times C_1.
\end{equation}
Each $L_i$ is a canonically oriented special Lagrangian two-torus, and for $i\neq j$,
\begin{equation}\label{eq:eisenstein-intersections}
 L_i\cap L_j=F.
\end{equation}

Let $x_1,y_1,x_2,y_2\in H^1(X;\Z)$ be the cohomology basis dual to the lattice coordinates on the two elliptic factors, oriented so that
\[
 q=x_1y_1x_2y_2,
 \qquad
 \int_Xq=1.
\]
For the first torus, an oriented parametrization is
\[
 (s,t)\longmapsto\bigl((1+2\omega)s,-(1+2\omega)t\bigr),
 \qquad (s,t)\in(\R/\Z)^2.
\]
Indeed, $1+2\omega=i\sqrt3$, so the pullback of $\operatorname{Re}\Omega$ is $3\,ds\wedge dt$. Writing $a_0=ds$ and $b_0=dt$, restriction gives
\[
 \iota_0^*(x_1,y_1,x_2,y_2)=(a_0,2a_0,-b_0,-2b_0).
\]
The defining adjunction
\[
 \int_X\beta\wedge(\iota_0)_*\eta
 =\int_{L_0}\iota_0^*\beta\wedge\eta
\]
determines the Gysin map by Poincar\'e duality. In particular it gives the class $P_0$ below. Transport by $r$, which preserves $\operatorname{Re}\Omega$ and permutes the three tori, gives $P_1$ and $P_2$. The Poincar\'e duals are
\begin{align}
 P_0&=(2x_1-y_1)(2x_2-y_2),\label{eq:P0}\\
 P_1&=(x_1+y_1)(x_2-2y_2),\label{eq:P1}\\
 P_2&=(x_1-2y_1)(x_2+y_2).\label{eq:P2}
\end{align}
A direct exterior-algebra computation gives
\begin{equation}\label{eq:P-relations}
 P_i^2=0,
 \qquad
 P_iP_j=9q\quad(i\neq j).
\end{equation}
In particular all nine local intersection signs in \eqref{eq:eisenstein-intersections} are $+1$.

For the odd-sector unit classes $e_i=1\in H^0(L_i)$, Theorem~\ref{thm:CY2-realization} therefore gives
\begin{equation}\label{eq:eisenstein-odd-odd-unreduced}
 e_i\star e_j
 =\sum_{p\in F}u_{i-j,p},
 \qquad i\neq j,
\end{equation}
where the rotation index is read modulo $3$.  The self-product is
\begin{equation}\label{eq:eisenstein-self-product}
 e_i^2=(\iota_i)_*(1)=P_i.
\end{equation}
More generally, if $\alpha\in H^*(L_i)$ and $\beta\in H^*(L_j)$ with $i\neq j$, then
\begin{equation}\label{eq:eisenstein-general-odd-odd}
 \alpha\star\beta
 =\sum_{p\in F}\alpha|_p\,\beta|_p\,u_{i-j,p}.
\end{equation}
Thus a product between distinct odd sectors vanishes whenever one of the inputs has positive ordinary cohomological degree.

For classes on the same odd torus one has the self-Gysin formula
\begin{equation}\label{eq:eisenstein-same-odd}
 \alpha\star\beta=(\iota_i)_*(\alpha\wedge\beta).
\end{equation}

\subsubsection{The invariant algebra}

We now pass from the unreduced $\widehat G$-graded algebra to its $\widehat G$-invariant subalgebra.  Via the standard orbit-sum map, this algebra is naturally identified with the direct sum over conjugacy classes of the corresponding centralizer-invariant sector cohomologies.  We keep the orbit-sum normalization because it makes the integral structure constants transparent.

The trace used below is the unreduced trace restricted to invariants.  If one instead uses the conventional quotient-stack integration on the identity sector, the trace and Frobenius pairing are multiplied by the global factor $|\widehat G|^{-1}=1/6$; the multiplication table is unchanged.

Define the invariant two-class
\begin{equation}\label{eq:Theta}
 \Theta
 =-2x_1x_2+x_1y_2+y_1x_2+y_1y_2.
\end{equation}
Then
\begin{equation}\label{eq:identity-invariants}
 H^*(X)^{\widehat G}
 =\R\langle1,\Theta,q\rangle,
 \qquad
 \Theta^2=6q.
\end{equation}
Here is an explicit derivation of the invariant space. On degree-one classes the generators act by
\[
\begin{aligned}
 r^*(x_1,y_1,x_2,y_2)&=(-y_1,x_1-y_1,-x_2+y_2,-x_2),\\
 \tau^*(x_1,y_1,x_2,y_2)&=(-x_1+y_1,y_1,-x_2+y_2,y_2).
\end{aligned}
\]
The action of $r$ has no invariant vector in degree one. Since the whole group preserves $q$, Poincar\'e duality gives no invariant vector in degree three either. In degree two the classes $x_1y_1$ and $x_2y_2$ are negated by $\tau$. For a cross-term
\[
 a\,x_1x_2+b\,x_1y_2+c\,y_1x_2+d\,y_1y_2,
\]
invariance under $r$ gives $b=c$ and $d=-a-c$, while invariance under $\tau$ gives $a=-2b=-2c$. Thus $c=d=b$ and the invariant line is spanned by $\Theta$. Expansion gives $\Theta^2=2(2+1)q=6q$. Together with degrees zero and four this proves \eqref{eq:identity-invariants}.

Equations \eqref{eq:P0}--\eqref{eq:P2} imply
\begin{equation}\label{eq:P-sum}
 P_0+P_1+P_2=-3\Theta.
\end{equation}

For $p\in F$ put
\begin{equation}\label{eq:Up}
 U_p=u_{1,p}+u_{2,p},
 \qquad
 W=\sum_{p\in F}U_p.
\end{equation}
Equation \eqref{eq:eisenstein-even-point-products} gives
\begin{equation}\label{eq:Up-product}
 U_pU_{p'}=2\delta_{p,p'}q.
\end{equation}

Choose an oriented basis $a_0,b_0\in H^1(L_0;\mathbb Z)$ with
\[
 \int_{L_0}a_0\wedge b_0=1,
\]
and transport it equivariantly to the other reflection tori by the $\widehat G$-action.  Thus the classes $a_i,b_i$ are conjugation-compatible, and with
\[
 t_i=a_i\wedge b_i
\]
we have $\int_{L_i}t_i=1$ for every $i$.  Let
\[
 E=\sum_{i=0}^2e_i,
 \qquad
 A=\sum_{i=0}^2a_i,
 \qquad
 B=\sum_{i=0}^2b_i,
 \qquad
 T=\sum_{i=0}^2t_i.
\]
These are orbit sums for the reflection conjugacy class.  Their shifted degrees are
\[
 |E|=1,
 \qquad |A|=|B|=2,
 \qquad |T|=3.
\]

\begin{theorem}[Explicit Eisenstein orientifold algebra]\label{thm:eisenstein-algebra}
For the action \eqref{eq:eisenstein-action}, the centralizer-invariant orientifold state space of Theorem~\ref{thm:CY2-realization} has dimension $16$.  With orbit-sum normalization it has basis
\[
 1,\Theta,q,
 \qquad
 \{U_p\}_{p\in F},
 \qquad
 E,A,B,T.
\]
Apart from the unit and products forced to vanish by degree or sector support, the multiplication is
\begin{align}
 \Theta^2&=6q,\label{eq:EA1}\\
 U_pU_{p'}&=2\delta_{p,p'}q,\label{eq:EA2}\\
 E^2&=3(W-\Theta),\label{eq:EA3}\\
 \Theta E=E\Theta&=-6T,\label{eq:EA4}\\
 U_pE=EU_p&=2T,\label{eq:EA5}\\
 ET=TE&=3q,\label{eq:EA6}\\
 AB&=3q,\qquad BA=-3q.\label{eq:EA7}
\end{align}
Moreover
\begin{equation}\label{eq:EAzero1}
 EA=AE=EB=BE=A^2=B^2=0,
\end{equation}
and
\begin{equation}\label{eq:EAzero2}
 \Theta U_p=\Theta A=\Theta B=\Theta T
 =U_pA=U_pB=U_pT=0.
\end{equation}
The trace is determined by $\operatorname{Tr}(q)=1$ and vanishes on the remaining basis elements.  The corresponding invariant pairing is nondegenerate, with basic nonzero values
\begin{align}
 \langle1,q\rangle&=\langle q,1\rangle=1,\nonumber\\
 \langle\Theta,\Theta\rangle&=6,\nonumber\\
 \langle U_p,U_{p'}\rangle&=2\delta_{p,p'},\nonumber\\
 \langle E,T\rangle&=\langle T,E\rangle=3,\nonumber\\
 \langle A,B\rangle&=3,\qquad
 \langle B,A\rangle=-3.\label{eq:EA-pairing}
\end{align}
Thus the algebra is Frobenius but is not graded-commutative with respect to the shifted sector grading.
\end{theorem}

\begin{proof}
The invariant identity sector has dimension three by \eqref{eq:identity-invariants}. The two nonidentity rotations are conjugate, their common fixed set has nine points, and $\tau$ fixes each of those points. Their invariant contribution therefore has basis $U_p$, $p\in F$, and dimension nine. The three reflections form one conjugacy class; the centralizer of $s_0$ is $\{1,s_0\}$ and acts pointwise on $L_0$. This contributes all of $H^*(L_0)$, with basis represented by the orbit sums $E,A,B,T$. The total dimension is consequently $3+9+4=16$.

The ambient relation is already proved. The inverse-point-sector products give \eqref{eq:EA2}. In the expansion of $E^2$, the three equal-label terms add to $P_0+P_1+P_2=-3\Theta$. Among the six unequal-label ordered pairs, three have product $r$ and three have product $r^2$. Equation~\eqref{eq:eisenstein-odd-odd-unreduced} therefore gives $3W$ for the remaining terms, proving \eqref{eq:EA3}.

The displayed restriction to $L_0$ gives
\[
 \iota_0^*\Theta=-6a_0b_0=-6t_0.
\]
Equivariance gives the same formula on each $L_i$, hence \eqref{eq:EA4}. For each output reflection label, there are exactly two rotation-reflection channels in $U_pE$: one from each of $u_{1,p}$ and $u_{2,p}$. Both are the Gysin map of the positively oriented point $p$ into that reflection torus, and each yields its normalized top class. Summing yields $U_pE=2T$. Reversing the inputs gives the same count, proving \eqref{eq:EA5}.

For $ET$, only equal-reflection labels contribute; each self-Gysin map sends $t_i$ to $q$. Thus $ET=TE=3q$. The identical argument applied to $a_i\wedge b_i=t_i$ and $b_i\wedge a_i=-t_i$ gives $AB=3q$ and $BA=-3q$. Products on different reflection tori with a positive-degree input vanish by restriction to points. The remaining contributions to $EA,AE,EB,BE$ lie in $H^3(X)^{\widehat G}=0$, and $A^2=B^2=0$ follows from the alternating wedge product. Degree and point-sector restriction give \eqref{eq:EAzero2} and all other stated zero products. Associativity follows from the surface theorem, whose product these formulas describe.

The trace pairing is the coefficient of $q$. Grouping the displayed basis into $(1,q)$, $\Theta$, the nine $U_p$, $(E,T)$ and $(A,B)$ gives blocks
\[
 \begin{pmatrix}0&1\\1&0\end{pmatrix},\quad
 (6),\quad 2I_9,\quad
 \begin{pmatrix}0&3\\3&0\end{pmatrix},\quad
 \begin{pmatrix}0&3\\-3&0\end{pmatrix}.
\]
Their determinants multiply to
\[
 \det\langle\ ,\ \rangle=(-1)\,6\,2^9\,(-9)\,9
 =248832\ne0.
\]
Thus the trace pairing is nondegenerate; associativity gives its Frobenius invariance.
\end{proof}

Two consequences of the multiplication are worth recording. From \eqref{eq:EA3}--\eqref{eq:EA6},
\[
 E^4=216q,
\]
and both bracketings give the same value.  Likewise
\[
 \Theta E^2=-18q=(\Theta E)E.
\]
These identities illustrate the surface sewing theorem.

\subsection{One reflection: a reusable multiplication formula}\label{subsec:one-reflection}

The preceding example has nontrivial even point sectors. The following examples instead take $G=\{1\}$ and $\widehat G=\langle\tau\rangle\cong C_2$, and distinguish disconnected reflection loci, free actions, and nonzero normal Euler classes. They are applications of the corrected surface theorem, not constructions of a different generalized cohomology theory. We keep real coefficients, the ordinary-form sign convention, and the trace normalized by $\operatorname{Tr}(q)=1$ for $\int_Xq=1$.

\begin{proposition}[A single reflection]\label{prop:one-reflection}
Let $X$ and $\tau$ satisfy the hypotheses of Theorem~\ref{thm:CY2-realization} with $G=\{1\}$. Write the fixed surface as $L=\coprod_{j=1}^m L_j$ and let $i_j:L_j\hookrightarrow X$. Then
\begin{equation}\label{eq:one-reflection-space}
 \mathcal B^k=H^k(X;\R)^\tau\oplus
 \bigoplus_{j=1}^m H^{k-1}(L_j;\R).
\end{equation}
For ambient invariant classes $a,b$ and reflection classes $\eta_j,\xi_j$, its complete product is
\begin{align}
 a\star b&=a\smile b,\nonumber\\
 a\star\eta_j&=i_j^*a\wedge\eta_j,\qquad
 \eta_j\star a=\eta_j\wedge i_j^*a,\label{eq:one-reflection-product}\\
 \eta_j\star\xi_k&=
 \begin{cases}(i_j)_*(\eta_j\wedge\xi_j),&j=k,\\0,&j\ne k.\end{cases}\nonumber
\end{align}
For the reflection-sector unit $e_j$, its normalized top class $t_j$, and $P_j=(i_j)_*1$, one has
\begin{equation}\label{eq:reflection-fourth-power}
 e_j^2=P_j,\quad e_jt_j=q,\quad
 P_je_j=-\chi(L_j)t_j,\quad
 e_j^4=P_j^2=-\chi(L_j)q.
\end{equation}
\end{proposition}
\begin{proof}
The involution restricts to the identity on each fixed component, so its ordinary conjugation invariants on $H^*(L_j)$ are all of that cohomology. Distinct components are disjoint. All obstruction bundles involving a reflection label have rank zero by Lemma~\ref{lem:CY2-sewing}; the remaining restriction and self-Gysin maps give \eqref{eq:one-reflection-product}. These maps preserve the specified orientations and give invariant ambient output. Theorem~\ref{thm:CY2-realization} gives associativity.

The first two identities in \eqref{eq:reflection-fourth-power} follow from the self-Gysin formula and $\int_X(i_j)_*t_j=1$. The self-intersection formula is
\[
 i_j^*P_j=e(N_{L_j/X})=-e(TL_j)=-\chi(L_j)t_j
\]
as an identity on the underlying surface, with the normal-orientation convention established after Lemma~\ref{lem:CY2-sewing}. Restriction gives $P_je_j$, and integration of $P_j^2$ gives the final identity. Here $t_j$ denotes either the normalized surface class or its copy in the shifted sector, according to the displayed target. This computation uses the actual normal Euler class and does not cancel a stable trivial summand.
\end{proof}

\subsection{Four reflection tori on a square abelian surface}\label{subsec:square-tori}

Real structures on abelian varieties belong to the classical Comessatti theory; see Silhol \cite{SilholComessatti}. The following explicit fixed-set and ring calculations are applications of the surface product, not a new classification of such real structures.

Put $E_i=\C/(\Z+i\Z)$, $X=E_i\times E_i$, and
\[
 \Omega=dz_1\wedge dz_2,\qquad
 \tau(z_1,z_2)=(\overline z_1,\overline z_2).
\]
Use the flat K\"ahler metric. With $z_k=x_k+i y_k$, the fixed locus has four connected components
\begin{equation}\label{eq:square-fixed-tori}
 L_\epsilon=\{y_1=\epsilon_1/2,\ y_2=\epsilon_2/2\},
 \qquad \epsilon=(\epsilon_1,\epsilon_2)\in\{0,1\}^2.
\end{equation}
Each is oriented by $dx_1\wedge dx_2=\operatorname{Re}\Omega|_{L_\epsilon}$. For simplicity, write $x_k,y_k$ also for the integral degree-one cohomology classes, and put
\[
 q=x_1y_1x_2y_2,\qquad P=-y_1y_2.
\]
The sign is fixed by $\int_Xx_1x_2P=1$. All four tori are homologous and have Poincar\'e dual $P$.

\begin{theorem}[A $24$-dimensional square-torus algebra]\label{thm:square-torus}
For this action the invariant ambient algebra is
\begin{equation}\label{eq:square-ambient}
 R_0=\Lambda_\R(x_1,x_2)\otimes\R[P]/(P^2),
 \qquad |x_1|=|x_2|=1,\quad |P|=2,
\end{equation}
and the full algebra has the vector-space decomposition
\begin{equation}\label{eq:square-ring}
 \mathcal B=R_0\oplus
 \bigoplus_{\epsilon\in\{0,1\}^2}\Lambda_\R(x_1,x_2)e_\epsilon.
\end{equation}
Here $|e_\epsilon|=1$. Besides the relations in $R_0$, its presentation is
\begin{equation}\label{eq:square-presentation}
 e_\epsilon x_k=x_ke_\epsilon,\qquad
 Pe_\epsilon=e_\epsilon P=0,\qquad
 e_\epsilon e_{\epsilon'}=\delta_{\epsilon,\epsilon'}P.
\end{equation}
Its shifted Poincar\'e polynomial and dimension are
\[
 P_{\mathcal B}(t)=1+6t+10t^2+6t^3+t^4,\qquad
 \dim_\R\mathcal B=24.
\]
The trace $\operatorname{Tr}(q)=1$ defines a nondegenerate Frobenius pairing.
\end{theorem}
\begin{proof}
Conjugation fixes $x_1,x_2$ and negates $y_1,y_2$. An ambient exterior monomial is invariant exactly when it contains an even number of the $y$-classes, proving \eqref{eq:square-ambient}. Its graded dimensions are $(1,2,2,2,1)$. Restriction to every $L_\epsilon$ kills the $y$-classes and retains $x_1,x_2$. The self-Gysin map sends $1$ to $P$ and, by the projection formula, sends $\eta$ to $\eta P$ for $\eta\in\Lambda(x_1,x_2)$. Formula \eqref{eq:one-reflection-product} now gives precisely \eqref{eq:square-presentation}. It also proves that these relations specify the entire product, rather than just a list of selected structure constants.

Each of the four shifted tori contributes $t+2t^2+t^3$. This proves the dimension and Poincar\'e polynomial. Associativity and nondegeneracy follow from Proposition~\ref{prop:one-reflection} and Corollary~\ref{cor:CY2-Frobenius}, or from the explicitly identified restriction and Gysin maps.
\end{proof}

Writing $a_\epsilon=x_1e_\epsilon$, $b_\epsilon=x_2e_\epsilon$ and $t_\epsilon=x_1x_2e_\epsilon$, some useful checks are
\[
 e_\epsilon^2=P\ne0,\qquad e_\epsilon^4=0,\qquad
 e_\epsilon t_\epsilon=q,\qquad
 a_\epsilon b_\epsilon=q,\quad b_\epsilon a_\epsilon=-q.
\]
All products between classes supported on different reflection components are zero. In particular, a nonzero self-Gysin class $P$ does not require a nonzero self-intersection number: here $e(N_{L_\epsilon/X})=0$ and $P^2=0$. The commutation rule in \eqref{eq:square-presentation} follows from the ordinary degree zero of $e_\epsilon$, not from its shifted degree one.

\subsection{A free affine real structure with the same volume phase}\label{subsec:free-affine}
On the same square abelian surface define
\begin{equation}\label{eq:free-affine}
 \tau_f(z_1,z_2)=(\overline z_1+\tfrac12,\overline z_2).
\end{equation}
Its square is translation by $(1,0)$ and therefore is the identity on $X$. It is anti-holomorphic, anti-symplectic and satisfies $\tau_f^*\Omega=\overline\Omega$. Nevertheless it has no fixed point: the first fixed-point equation would imply
\[
 2i y_1=\tfrac12+m+in\quad\text{for some }m,n\in\Z,
\]
whose real part is impossible. Since translation is isotopic to the identity, $\tau_f$ and $\tau$ induce the same action on ambient cohomology. Thus
\begin{equation}\label{eq:free-affine-ring}
 \mathcal B_f=R_0,\qquad
 P_{\mathcal B_f}(t)=1+2t+2t^2+2t^3+t^4,\qquad
 \dim_\R\mathcal B_f=8.
\end{equation}
There are no reflection states and no nonidentity fixed sectors. Equivalently, after rotation this is the ordinary real cohomology ring of the free quotient. The phase obstruction vanishes for both \eqref{eq:square-fixed-tori} and \eqref{eq:free-affine}; its vanishing says nothing about nonemptiness of the fixed locus. These two actions also show that identical actions on ambient cohomology and on the volume line need not give the same full sector algebra.

There is a distinction from an origin-preserving real abelian variety: $\tau_f(0)=(\tfrac12,0)\ne0$. Thus $\tau_f$ is an affine real structure on the underlying complex algebraic surface, not a real group structure for the specified origin. Since it has no fixed point, no change of origin makes that origin fixed. Origin-preserving classification results such as those discussed in \cite{SilholComessatti} must not be applied without this qualification.

\subsection{Three real quartic K3 models}\label{subsec:quartic-k3}

The next examples leave the torus setting and include a nonzero normal Euler class. We use standard K3 topology and the real-locus orientation convention; see \cite{HuybrechtsK3,ItenbergMikhalkin}. The general theory of real structures and integral quadratic forms is much richer \cite{NikulinReal}. Nikulin--Saito \cite{NikulinSaito} study real K3 surfaces equipped additionally with a commuting holomorphic non-symplectic involution, with hypotheses for their moduli classification. Those additional data are not assumed in our single-reflection calculation. No integral-lattice classification or new classification of real quartics is claimed below.

\subsubsection{The cohomology and product for one connected real component}
Let $X$ be a K3 surface with anti-holomorphic involution $\tau$. Choose an anti-invariant K\"ahler form and normalize the holomorphic two-form so that $\tau^*\Omega=\overline\Omega$. Such choices satisfy the single-reflection surface hypotheses. Set
\[
 V=H^2(X;\R)^\tau,\qquad r=\dim V,\qquad
 Q(v,w)=\int_X v\smile w.
\]
The real K3 facts used here are $b_2(X)=22$, $H^1(X)=H^3(X)=0$, and signature $(3,19)$ of the full intersection form \cite{HuybrechtsK3}. The finite-order Lefschetz formula gives
\begin{equation}\label{eq:k3-invariant-rank}
 \chi(X^\tau)=2+(r-(22-r))=2r-20,
 \qquad r=10+\tfrac12\chi(X^\tau).
\end{equation}
Indeed, $\tau$ preserves the ambient real orientation, and the normal action at its fixed surface is $-1$, so the fixed contribution is its Euler characteristic. The invariant and anti-invariant subspaces are $Q$-orthogonal. The positive classes $\operatorname{Re}\Omega$, $\operatorname{Im}\Omega$, and a K\"ahler class have respectively invariant, anti-invariant, and anti-invariant parity. Hence $Q|_V$ is nondegenerate of signature $(1,r-1)$; compare \cite[Proposition~2.1]{ItenbergMikhalkin}.

For an arbitrary real locus $L$, the shifted dimensions consequently satisfy
\begin{equation}\label{eq:k3-poincare}
 P_{\mathcal B}(t)=1+b_0(L)t+(r+b_1(L))t^2+b_0(L)t^3+t^4.
\end{equation}
In particular an empty real locus gives $r=10$ and a $12$-dimensional algebra consisting only of the invariant ambient cohomology.

\begin{proposition}[A real normal form for connected K3 reflection loci]\label{prop:k3-normalform}
Suppose $L=X^\tau$ is connected of genus $g$, and put $P=i_*1\in V$. Let $e,t$ be the shifted copies of $1\in H^0(L)$ and the class of integral one in $H^2(L)$. Choose a symplectic basis of $H^1(L)$, and denote its shifted copies by $A_\ell,B_\ell$ for $1\leq\ell\leq g$. The state space is
\[
 \R1\oplus V\oplus\R q\oplus
 \R e\oplus\R t\oplus
 \bigoplus_{\ell=1}^g(\R A_\ell\oplus\R B_\ell),
\]
with shifted degrees $0,2,4,1,3,2$, respectively. The complete nonunit products are
\begin{align}
 vw&=Q(v,w)q,&e^2&=P,&ve=ev&=Q(v,P)t,\label{eq:k3-ring}\\
 et=te&=q,&A_\ell B_m&=\delta_{\ell m}q,&B_mA_\ell&=-\delta_{\ell m}q.\nonumber
\end{align}
Every nonunit product not specified by these formulas is zero. Moreover
\begin{equation}\label{eq:k3-P-square}
 P\ne0,\qquad Q(P,P)=2g-2,\qquad e^4=(2g-2)q.
\end{equation}
These are presentations over $\R$; a real normal-form basis below is not asserted to be an integral K3 lattice basis.
\end{proposition}
\begin{proof}
The pairing $Q$ specifies the ambient ring because there is no odd ambient cohomology. For $v\in V$, the integral of $i^*v$ is $Q(v,P)$, which determines this degree-two restriction on the connected surface. The Gysin map on $H^1(L)$ is zero because its target is $H^3(X)=0$; the Gysin map on $H^2(L)$ is integration times $q$. Proposition~\ref{prop:one-reflection} gives every formula in \eqref{eq:k3-ring}, including the displayed zero products and the signs on $H^1(L)$. It also gives $Q(P,P)=-\chi(L)=2g-2$. Finally
\[
 Q([\operatorname{Re}\Omega],P)=\int_L\operatorname{Re}\Omega>0,
\]
so $P$ is nonzero, including when $g=1$. This is the usual real-locus area orientation \cite{ItenbergMikhalkin}. The fourth-power identity follows from the actual normal self-intersection.
\end{proof}

\subsubsection{Explicit signed quartics and their fixed sets}
Consider the following smooth quartics, each with standard conjugation:
\begin{align}
 Y_S&=\{z_0^4-z_1^4-z_2^4-z_3^4=0\}\subset\mathbb{CP}^3,\label{eq:quartic-sphere}\\
 Y_T&=\{z_0^4+z_1^4-z_2^4-z_3^4=0\}\subset\mathbb{CP}^3,\label{eq:quartic-torus}\\
 Y_\varnothing&=\{z_0^4+z_1^4+z_2^4+z_3^4=0\}\subset\mathbb{CP}^3.\label{eq:quartic-empty}
\end{align}
Smoothness follows since their four partial derivatives vanish simultaneously only at the excluded origin. Adjunction and the Lefschetz hyperplane theorem show that each is a K3 surface. For completeness, their total Chern class is
\[
 c(TY)=\frac{(1+H)^4}{1+4H}=1+6H^2,
 \qquad \int_YH^2=4,
\]
so $\chi(Y)=24$ and $b_2(Y)=22$. The restricted Fubini--Study form is anti-invariant under conjugation, and a residue holomorphic two-form can be normalized to obey $\tau^*\Omega=\overline\Omega$.

These three surfaces are isomorphic over $\C$: diagonal changes of coordinates using a fourth root of $-1$ change the indicated coefficient signs. Their real structures are different. Their real loci can be determined without a real-quartic classification. For $Y_S$, the real coordinate $z_0$ cannot vanish. In the affine chart $z_0=1$, the fixed locus is
\[
 x_1^4+x_2^4+x_3^4=1,
\]
a smooth radial graph over $S^2$, and hence a sphere. For $Y_T$, write the two coordinate pairs as $u=(z_0,z_1)$ and $v=(z_2,z_3)$. A real projective solution can be normalized to
\[
 u_0^4+u_1^4=v_0^4+v_1^4=1.
\]
Each factor is a smooth circle, with the remaining identification $(u,v)\sim(-u,-v)$. The quotient is a torus, since this simultaneous antipodal action is translation by a half-period on $S^1\times S^1$. For $Y_\varnothing$, a sum of four real fourth powers cannot vanish at a projective point, so the fixed locus is empty.

\begin{example}[A spherical component with nonzero Euler correction]\label{ex:k3-sphere}
For $Y_S$, equations \eqref{eq:k3-invariant-rank}--\eqref{eq:k3-poincare} give
\[
 r=11,\qquad P_{\mathcal B_S}(t)=1+t+11t^2+t^3+t^4,
 \qquad \dim\mathcal B_S=15.
\]
One may choose a real basis $P,h,n_1,\ldots,n_9$ of $V$ with orthogonal form
\[
 Q=\operatorname{diag}(-2,1,-1,\ldots,-1).
\]
Together with \eqref{eq:k3-ring}, this is a complete real algebra presentation. In particular,
\begin{equation}\label{eq:k3-spherical-products}
 e^2=P,\qquad Pe=-2t,\qquad P^2=-2q,\qquad et=q,\qquad e^4=-2q.
\end{equation}
The normal Euler class on the sphere is $-2$ times its normalized top class and is nonzero. Indeed the underlying normal bundle is isomorphic to $TS^2$, so it is stably trivial:
\[
 N_{L/Y_S}\oplus\underline\R\cong_{\mathrm{or}}\underline\R^3,
 \qquad e(N_{L/Y_S})=-2u\ne0.
\]
The chosen normal orientation accounts for the minus sign; stabilization can be oriented compatibly by an oriented trivialization. Thus the vanishing of the binary orbifold obstruction bundle does \emph{not} make every excess/self-intersection correction vanish. This example tests precisely the ordinary Euler information that cannot be read from stable real bundle data alone.
\end{example}

\begin{example}[A torus component on a K3 surface]\label{ex:k3-torus}
For $Y_T$ the invariant ambient rank is $r=10$, and
\[
 P_{\mathcal B_T}(t)=1+t+12t^2+t^3+t^4,\qquad
 \dim\mathcal B_T=16.
\]
Here $P\ne0$ is isotropic. Choose $D\in V$ with $Q(P,D)=1$ and $Q(D,D)=0$, and complete $P,D$ to a basis by eight mutually orthogonal vectors of square $-1$. Such a choice follows from signature $(1,9)$: after first choosing $Q(P,D)=1$, replace $D$ by $D-\tfrac12Q(D,D)P$. The complete real presentation is \eqref{eq:k3-ring} for this hyperbolic plane and negative complement, with one pair $A,B$. It includes
\begin{equation}\label{eq:k3-torus-products}
 e^2=P\ne0,\quad P^2=Pe=0,\quad De=t,\quad
 et=q,\quad AB=q,\quad BA=-q,\quad e^4=0.
\end{equation}
\end{example}

\begin{example}[A K3 surface without real points]\label{ex:k3-empty}
The specific quartic $Y_\varnothing$ already appears as an example without real points in \cite[Introduction]{YoshikawaRealK3}; we use that known real surface to illustrate the present ring formulas. For $Y_\varnothing$, there are no reflection states. The invariant ambient rank is $10$, so the algebra has dimension $12$ and Poincar\'e polynomial $1+10t^2+t^4$. It is specified by $\R1\oplus V\oplus\R q$ with $vw=Q(v,w)q$ and $Q$ of signature $(1,9)$. As with the free affine torus, the phase obstruction vanishes although the reflection sector is empty.
\end{example}

\paragraph{The free quotient and the Enriques interpretation.}
Let $M=Y_\varnothing/\langle\tau\rangle$ and let $\pi:Y_\varnothing\to M$ be the two-sheeted covering. Since the action is free, transfer gives the ordinary real algebra identification
\[
 \pi^*:H^*(M;\R)\stackrel{\cong}{\longrightarrow}
 H^*(Y_\varnothing;\R)^\tau=\mathcal B.
\]
Using compatible invariant Ricci-flat K\"ahler data, Donaldson's trick makes $\tau$ holomorphic for an integrable rotated structure; its quotient is an Enriques surface \cite[Sections~3--4]{YoshikawaRealK3}. This is an additional interpretation with those metric choices, not a claim that every $J_\Omega$ in Lemma~\ref{lem:CY2-sewing} is integrable. The ring identification by transfer does not require integrability. For a top-degree class $a$ on $M$, the manuscript's trace satisfies $\operatorname{Tr}(\pi^*a)=2\int_M a$; the quotient-stack trace, scaled by $1/2$, equals quotient integration.

\begin{corollary}[Equal graded dimensions do not determine the algebra]\label{cor:two-sixteen-dimensional}
The algebra of $Y_T$ and the Eisenstein algebra of Theorem~\ref{thm:eisenstein-algebra} have the same shifted Poincar\'e polynomial $1+t+12t^2+t^3+t^4$, but are not isomorphic as graded real algebras.
\end{corollary}
\begin{proof}
The degree-one space is one-dimensional in each algebra. On $Y_T$ it is generated by $e$ with $e^4=0$. In the Eisenstein algebra it is generated by $E$ with $E^4=216q\ne0$. A graded isomorphism would send $e$ to a nonzero real multiple of $E$, contradicting the fourth-power relation. This distinction is independent of trace normalization.
\end{proof}

The five additions are summarized below. Each uses the full ordinary conjugation-invariant surface algebra, not the closed-sector anti-linear real form of Theorem~\ref{thm:closed-real}.
\begin{center}
\small
\begin{tabular}{lccc}
Example & reflection locus & shifted dimensions $(b_0,\ldots,b_4)$ & dimension\\ \hline
Square conjugation & $4T^2$ & $(1,6,10,6,1)$ & $24$\\
Free affine square action & $\varnothing$ & $(1,2,2,2,1)$ & $8$\\
Signed quartic $Y_S$ & $S^2$ & $(1,1,11,1,1)$ & $15$\\
Signed quartic $Y_T$ & $T^2$ & $(1,1,12,1,1)$ & $16$\\
Definite quartic $Y_\varnothing$ & $\varnothing$ & $(1,0,10,0,1)$ & $12$
\end{tabular}
\end{center}
These examples vary the reflection geometry and the resulting product, rather than merely the notation for the original dihedral example. They do not enlarge the surface theorem to arbitrary complex dimension, prove an integral lattice classification, or assert priority for the underlying real forms.

\section{Comparison and further directions}\label{sec:comparison}

\subsection{Unoriented field theories and the scope of the construction}
The organization of group-labelled sectors into Frobenius structures, followed by passage to invariants, predates the present construction; see Kaufmann \cite{KaufmannFrobenius}. For unoriented theories, Turaev and Turner relate two-dimensional unoriented topological quantum field theories to extended Frobenius algebras \cite{TuraevTurner}. Tagami treats unoriented homotopy quantum field theories with target $K(\pi,1)$ when every element of $\pi$ has square one, using extended crossed group algebras \cite{Tagami}. Equivariant and orientation-twisted constructions are developed by Sweet \cite{Sweet} and Young \cite{YoungDW}.

The systematic study of extended Frobenius structures in \cite{CKQWExtended} makes the extra algebraic data explicit: an involution of the Frobenius algebra and a distinguished element subject to compatibility identities. Neither follows merely from a nondegenerate invariant pairing. Applying an ungraded extension theorem would also require a comparison with the sector grading and sign conventions used here.

Our construction specifies sector cohomology and binary products; it does not supply all unoriented cobordism operations or establish an equivalence with one of these theories. Such a comparison would require the state spaces, coefficient systems, crosscap operations and sewing axioms to be matched. In particular, no disjointness or novelty conclusion relative to those theories is inferred solely from the presence of antiholomorphic fixed loci.

The group labels here include the fixed loci of anti-complex involutions. The relation $\tau g\tau=g^{-1}$ need not define a commuting direct-product action. This explains the usefulness of the graded-group notation, but is not by itself evidence for a new cohomology theory outside known orbifold constructions.

\subsection{The Chen--Hu real form and the rotated surface ring}
There are two distinct algebraic operations. Theorem~\ref{thm:closed-real} constructs a phase-corrected anti-linear real form of the $G$-invariant even Chen--Ruan algebra. Theorem~\ref{thm:CY2-realization} instead takes ordinary conjugation invariants of the full group-labelled surface algebra. Its even invariant part need not be the real form of the first theorem. Keeping these distinctions avoids conflating label reversal with conjugation descent.

For the effective phase-trivial surface action, the full invariant product is an ordinary Chen--Ruan product for $(X,J_\Omega,\widehat G)$. Its existence and associativity therefore follow from established orbifold theory \cite{ChenRuan,FantechiGottsche}. The geometric fact that an antiholomorphic involution can become holomorphic after rotation is also established in the hyper-K\"ahler setting \cite{BiswasWilkin}. The explicit pointwise construction and obstruction-rank calculation in Lemma~\ref{lem:CY2-sewing} explain why the formulas proposed here have precisely that interpretation.

On a nonabelian group-labelled inertia, exchanging inputs is naturally a braided operation rather than the naive interchange of labels. The map $(g,h,x)\mapsto(ghg^{-1},g,x)$ is the one used in \cite[Proposition~3.1.3]{EJKPlethora}. Our formulas first use ordered labels and only then take conjugation invariants.

This identification also explains the shifted parity phenomenon in the Eisenstein algebra. The reflection sectors have age $1/2$ for $J_\Omega$, so their degree shift is odd. Ordinary wedge parity and total shifted parity need not agree. Thus the facts that $E^2\ne0$ with $|E|=1$, and $AB=-BA\ne0$ with $|A|=|B|=2$, do not establish a new source of noncommutativity beyond the age-shift conventions of the rotated orbifold ring.

\subsection{Comparison with virtual intersection products}\label{subsec:virtual-comparison}
Virtual orbifold cohomology \cite{LUXVirtual} uses the Euler class of the excess bundle of two fixed loci as its binary correction. Gonz\'alez, Lupercio, Segovia, Uribe and Xicot\'encatl identify this algebra with the Chen--Ruan cohomology of the cotangent orbifold \cite[Theorem~1.1]{GLSUXCotangent}. It is important to distinguish this cotangent comparison from our surface rotation.

Here is a direct calculation using the definitions in \cite[Section~2]{GLSUXCotangent} and Proposition~\ref{prop:one-reflection}. Work on the \emph{rotated} almost complex surface $(X,J_\Omega)$ with one reflection $\tau$, and fix a connected component $i:L\hookrightarrow X$. For the virtual product its self-pair excess bundle and shift are
\[
 E(X;L,L)=TX|_L/TL=N_{L/X},\qquad
 d^{\mathrm{virt}}_\tau=2\operatorname{codim}_{\C}(L,X)=2.
\]
For the product $\star$ of this paper, the obstruction rank is zero and $d_\tau=2\age_{J_\Omega}(\tau)=1$. Write $\mathbf 1_L^{\mathrm{virt}}$ and $e_L$ for the copies of the ordinary unit of $H^0(L)$ in these differently graded theories. With the same geometric orientations and unreduced convention,
\begin{align}
 \mathbf 1_L^{\mathrm{virt}}\mathbin{\times_{\mathrm{virt}}}
       \mathbf 1_L^{\mathrm{virt}}&=i_*e(N_{L/X}),\label{eq:virtual-reflection-square}\\
 e_L\star e_L&=i_*1=P_L,\qquad
 e_L^{\star4}=P_L^2=i_*e(N_{L/X}).\label{eq:rotated-reflection-comparison}
\end{align}
The last equality is the self-intersection and projection formula, not an identification of the two products. For the spherical K3 example, the virtual square is $-2q$, whereas the present square is $P_L\ne0$ and its fourth power is $-2q$. For the torus examples, the virtual square vanishes while $e_L\star e_L=P_L\ne0$. Thus their shifts and binary corrections differ even on the same fixed-set cohomology. No identification under the natural fixed-locus correspondence is asserted, and no virtual excess factor is inserted into the surface multiplication tables above.

\subsection{What the phase obstruction does and does not imply}
The class $[\mu_\R]$ is the negative Bockstein of the volume-phase cocycle. Its vanishing is equivalent to coherent real lifts and to phase-trivial normalization, and the lift classification is the indicated group-cohomological torsor. These statements follow from the finite-group phase law and averaging, independently of any analytic construction. They do not assert that all higher-dimensional correction bundles or Euler comparisons exist. The no-go theorem, in turn, rules out a conjugacy-invariant single-sector age satisfying the defect rule for all reflection pairs; it does not prohibit every possible grading after additional structure has been chosen.

In particular, the degree distinction from orientifold discrete torsion \cite{SharpeDiscrete} is
\[
 [a]\in H^1(\widehat G;U(1)_\epsilon),\qquad
 -\beta[a]\in H^2(\widehat G;\Z_\epsilon),\qquad
 [\omega_{\mathrm{DT}}]\in H^2(\widehat G;U(1)_\epsilon).
\]
The first two classes here come from the volume line; the last denotes a discrete torsion class, not the K\"ahler form. We do not identify these classes or deduce $B$-field, crosscap, or Real Gromov--Witten operations from volume normalization.

\subsection{Relation to analytic Fredholm sewing}
The analytic study of Real Cauchy--Riemann families, determinant lines and $KO$-valued sewing is a separate problem. The present proofs do not invoke the author's companion manuscript on that subject. A stable family index identity can motivate correction data, but cannot replace the actual Euler class comparison \eqref{eq:euler-sewing}. To feed an analytic construction into Definition~\ref{def:admissible-package}, one must prove that it supplies the required finite-dimensional bundles, coefficient identifications, clean comparison maps and Euler equality, with the stated parity conventions or a separately justified sign extension.

\subsection{Complex stringy K-theory and the Real boundary}\label{subsec:HuWang-comparison}
There is an existing complex $K$-theory comparison for the surface algebra, whereas a Real or twisted $KR$ comparison is a separate question. To make the distinction precise, put
\[
 \mathfrak Y=[(X,J_\Omega)/\widehat G],\qquad
 K^*_{\mathrm{orb}}(\mathfrak Y;\C)
 :=K^*_{\mathrm{orb}}(\mathfrak Y)\otimes_{\Z}\C.
\]
Here $K^*=K^0\oplus K^1$ is ordinary two-periodic complex orbifold $K$-theory. For this finite global quotient its underlying group is $K^*_{\widehat G}(X)\otimes_{\Z}\C$, with complex-linear equivariance on the bundle fibers. In particular, it is the $K$-theory of the base orbifold $\mathfrak Y$, not the full orbifold $K$-theory of its inertia. Write $\circ_{\mathrm{HW}}$ for the product of \cite[Definition~4.3]{HuWang}, and $\widetilde{\operatorname{ch}}_{\mathrm{deloc}}$ for the modified character of \cite[Definition~4.4]{HuWang}. The product is in general not the ordinary tensor product, and the modification of the character is part of the comparison.

\begin{corollary}[Hu--Wang comparison after rotation]\label{cor:HuWang-rotation}
Under the hypotheses of Theorem~\ref{thm:CY2-realization}, composition of its rotation identification with the Hu--Wang character gives an isomorphism of complex algebras
\begin{equation}\label{eq:HuWang-rotation}
 \Psi_\Omega:
 \bigl(K^*_{\mathrm{orb}}(\mathfrak Y;\C),\circ_{\mathrm{HW}}\bigr)
 \xrightarrow{\ \cong\ }
 \bigl(H^*_{\Ori}(X,\widehat G)\otimes_{\R}\C,\star\bigr).
\end{equation}
The two-periodic parity on the right is ordinary fixed-sector cohomological parity, before the age shifts. No preservation of the full shifted integer grading is asserted.
\end{corollary}
\begin{proof}
The group $\widehat G$ preserves $J_\Omega$ by Lemma~\ref{lem:CY2-sewing}, and $X$ is compact. Its action is effective: an even element in the kernel is excluded by the effective $G$-action, and an odd element cannot act as the identity because its differential anticommutes with the original complex structure $J$. Thus $\mathfrak Y$ is a compact effective almost complex global quotient orbifold. Hu--Wang \cite[Theorem~4.5]{HuWang} therefore applies and gives a ring isomorphism
\[
 \widetilde{\operatorname{ch}}_{\mathrm{deloc}}:
 \bigl(K^*_{\mathrm{orb}}(\mathfrak Y;\C),\circ_{\mathrm{HW}}\bigr)
 \longrightarrow \bigl(H^*_{\CR}(\mathfrak Y;\C),\star_{\CR}\bigr).
\]
Let $\mathcal I_\Omega$ be the complexification of the ring identification in Theorem~\ref{thm:CY2-realization}. Then
\[
 \Psi_\Omega=\mathcal I_\Omega^{-1}\circ
       \widetilde{\operatorname{ch}}_{\mathrm{deloc}}
\]
is the required isomorphism. The delocalized character uses even and odd ordinary cohomology on inertia, and its modifying characteristic class has even ordinary degree. The map $\mathcal I_\Omega$ uses the same fixed-locus cohomology. Hence the comparison preserves that two-periodic parity, not necessarily the parity of the age-shifted degree.
\end{proof}

This corollary is an application of the existing Hu--Wang theorem, not an independent construction of their product or Chern character. It applies to the surface examples above without changing their multiplication tables. A reflection unit has ordinary degree zero and shifted degree one; this already illustrates why the two parity conventions cannot be conflated. Nor does \eqref{eq:HuWang-rotation} identify the closed-sector real form $\Fix(\RR)$ with the full conjugation-invariant surface ring. No integral, torsion-sensitive, or Real-equivariant refinement is inferred from complexification.

\paragraph{\textbf{Twisted pushforwards and differential refinements.}}
Carey--Wang \cite{CareyWangThom} construct Thom isomorphisms and wrong-way maps in twisted complex $K$-theory, with the source twisting adjusted by the obstruction to an ordinary $K$-orientation. Carey--Mickelsson--Wang \cite{CareyMickelssonWang} construct differential twisted $K$-theory and a twisted Chern character depending on a connection and curving on the twisting gerbe, together with a Riemann--Roch comparison. These are relevant sources for the additional twist and orientation data of a future refinement; they do not identify those data with the coefficient lines in Definition~\ref{def:admissible-package}. The corollary above uses the untwisted theorem of Hu--Wang; the suggested torsion-twisted extension discussed in \cite[Section~5]{HuWang} is not an input here.

Atiyah's $KR$-theory \cite{AtiyahKR}, twisted equivariant and orbifold $K$-theory \cite{AdemRuanTwistedK,LupercioUribeGerbes,TuXuLaurent}, and Real bundle-gerbe and twisting constructions \cite{HekmatiMurraySzaboVozzo,HekmatiSigns,LudersOttoWaldorf,FreedMoore} provide natural contexts for anti-linear symmetry and coefficient data. No isomorphism
\[
 KR_{\widehat G}(X)\otimes\R\cong H^*(\AOri)
\]
is asserted here. Such a comparison requires a precise twisting, Chern character and Real orientation convention. In particular, replacing the ordinary Euler class by a generalized-cohomology Euler class changes the construction and must not be used retrospectively to justify the invalid stable cancellation.

\subsection{Real Gromov--Witten theory and higher operations}
Orienting Real Cauchy--Riemann determinant lines is a global problem, as the constructions in Real Gromov--Witten theory demonstrate \cite{GeorgievaZinger}. A moduli-space extension requires compatible evaluation maps and compactifications, together with correction classes and orientations. The surface ring in this paper is a constant-sector ordinary orbifold ring after rotation; it is not identified with a genus zero Real Gromov--Witten algebra. The analysis of disc orientations under anti-symplectic involutions in \cite{FOOOAnti} requires additional data of relative-spin and moduli space; it does not identify this constant-sector ring with Floer cohomology. Likewise, the determinant line constructions in \cite{ZingerDet} fix analytic coherence conventions, not the ordinary Euler equality of two total correction bundles. The present cohomological proof does not automatically give a strict differential graded or $A_\infty$ enhancement.

\subsection{Gerbes, transgression and further comparison}
Jandl structures, unoriented gerbe holonomy and discrete torsion \cite{SchreiberSchweigertWaldorf,GawedzkiSuszekWaldorfWZW,GawedzkiSuszekWaldorfGerbes,SharpeDiscrete} address Wess--Zumino and equivariant $B$-field data under reversal. Loop groupoid and reflection transgression frameworks \cite{LupercioUribeLoop,NoohiYoung,YoungDW} explain why inversion of holonomy labels naturally occurs. The class in Section~\ref{sec:odd} is specifically a volume-line phase Bockstein, not an identification with a gerbe obstruction. An intrinsic stack or groupoid formulation of the correspondence data would need its own descent and Morita-invariance proof; those properties are not claimed from the global quotient formulas alone.

\subsection{Conclusion}
The established conclusions are a cohomological closed-sector real form with the necessary invariants, a precise single-sector age obstruction, a volume-line normalization criterion, an Euler-admissible even-rank realization theorem, and an explicit surface identification with rotated Chen--Ruan theory. The Eisenstein calculation illustrates the latter with a completely specified multiplication and trace. The square-torus and real quartic examples test disconnected and empty reflection loci, nonzero normal Euler classes, and a ring distinction not visible in graded dimensions. For the rotated surface quotient, Corollary~\ref{cor:HuWang-rotation} also records the ordinary complex stringy $K$-theory comparison supplied by Hu--Wang. Higher-dimensional existence, arbitrary odd-rank signs, analytic realization, Real $KR$ refinements and full unoriented field-theory operations remain distinct problems rather than consequences of these results.

\section*{Acknowledgements}
The author would like to thank Bohui Chen, An-Min Li, and Guosong Zhao for their constant support and encouragement.  Special thanks are due to Sung-Soo Kim, Kimyeong Lee, Futoshi Yagi, Satoshi Nawata, and Rui-Dong Zhu for many helpful and stimulating discussions.  In particular, the author is grateful to Andrea Brini, Todor Milanov, Arpan
Saha for their inspiring online seminars delivered during the pandemic, and to Doan Nhat Minh
and Van Nguyen for the enjoyable and fruitful collaboration. The author also gratefully acknowledges the many friends and colleagues met at various conferences for their friendship, encouragement, and stimulating conversations.  Finally, the author would also like to express special thanks to the Mainz Institute for Theoretical Physics (MITP) of the Cluster of Excellence PRISMA$^{+}$ (Project ID~390831469) for its hospitality and support.

\section*{Statements and Declarations}
\noindent\textbf{Funding.}
This work was supported by the National Natural Science Foundation of China (NSFC) under Grant Nos.~11501470, 11426187, and 11791240561, and was partially supported by NSFC Grant No.~11671328, the Chengdu Science and Technology Program under Grant No.~2025-YF09-00007-SN, and the Fundamental Research Funds for the Central Universities under Grant Nos.~2682021ZTPY043, 2682025ZTPY001, 2682025ZTPY057, and 2682025ZTO002.

\medskip
\noindent\textbf{Competing Interests.}
The author declares that there are no competing interests relevant to the content of this article.

\medskip
\noindent\textbf{Data availability.}
This article is theoretical. The proofs and explicit algebraic calculations are contained in the text; no empirical data are used.

\appendix
\section{Determinant lines, Euler classes and excess intersection}\label{app:det}

We collect the orientation conventions used in Section~\ref{sec:admissible}; see \cite{BottTu,Fulton} for Thom and excess intersection background.  Let $E\to M$ be a real vector bundle of rank $r$.  Its determinant line is
\[
 \det(E)=\Lambda^r_{\R}E,
\]
and its orientation local system is denoted by $o(E)$.  There is a canonical functorial isomorphism
\begin{equation}\label{eq:det-sum}
 \det(E\oplus F)\cong\det(E)\otimes\det(F),
\end{equation}
with the usual Koszul convention when factors are permuted.  In the paper all determinant line identifications are ordered according to the written direct sums, so the order is part of the specified comparison data.

If $E$ is a real vector bundle, its Euler class with local coefficients is
\[
 e(E)\in H^r(M;o(E)).
\]
For oriented direct sums, the Whitney formula gives
\begin{equation}\label{eq:euler-whitney}
 e(E\oplus F)=e(E)\smile e(F),
\end{equation}
under the identification $o(E\oplus F)\cong o(E)\otimes o(F)$.  A de Rham representative may be chosen using a Thom form on the total space and pullback by the zero section.  This is the meaning of the closed Euler forms $\mathbf e_{\gamma,\delta}$ in Definition~\ref{def:admissible-package}.

Let $i:Y\hookrightarrow M$ be a proper embedding of real codimension $c$.  A Gysin map with local coefficients has the form
\[
 i_*:H^k(Y;i^*\mathcal L\otimes o(N_i))\longrightarrow H^{k+c}(M;\mathcal L).
\]
Thus the factor $o(N_m)$ on the right of \eqref{eq:line-compatibility} is forced by the coefficient type of the pushforward.

For a tubular neighborhood with projection $\pi$ and a vertically compact Thom form $\Phi_N$, use the representative $\pi^*\eta\wedge\Phi_N$ and extension by zero. With the even codimension convention in Section~\ref{sec:admissible}, this commutes with $d$ in the convention of \eqref{eq:ordinary-leibniz}. Its cohomology map is independent of the tubular and Thom choices. This construction does not assert strict associativity of all representatives. For the related de Rham identities up to homotopy in global quotient stringy theory, see \cite[Sections~4--5, especially Theorem~4.9]{KaufmannDeRham}.

For a clean square, the oriented base change formula inserts the Euler class of the actual excess bundle. Applying it in the two triple compositions gives $e(B_L)$ and $e(B_R)$ with the written direct sum order. The formal theorem requires the explicit comparison \eqref{eq:euler-sewing}, together with the specified compatibility of the final pushforwards. The parity hypotheses make the correction permutations sign-free; an odd rank version would have to track them explicitly.

A stable oriented real bundle isomorphism does not imply equality of ordinary Euler classes, even if representatives have been chosen. The counterexample
\[
 TS^2\oplus\underline\R\cong\underline\R^3,
 \qquad e(TS^2)=2u\ne0=e(\underline\R^2)
\]
also shows why stabilization by a trivial line cannot be cancelled in the Whitney formula. The sufficient condition is an actual compatible isomorphism before stabilization, or a direct proof of the required Euler equality. No claim about ordinary Euler classes follows merely from equality in $KO$.

\end{document}